\documentclass[11pt]{article}

\usepackage{graphicx,amsthm}
\usepackage{xspace}
\usepackage{xcolor} 
\usepackage{hyperref} 
\usepackage{amssymb,amsmath,enumerate,bbold}
\usepackage[mathscr]{euscript}
\usepackage[all,2cell]{xy}
\UseAllTwocells

\newdir{|>}{!/4.5pt/@{|}*:(1,-.2)@^{>}*:(1,+.2)@_{>}}

\mathcode`\:="603A     % : = \colon
\mathcode`\<="4268     % < = \langle
\mathcode`\>="5269     % > = \rangle
\mathchardef\gt="313E  % arithmetic
\mathchardef\lt="313C  % strict order
\mathchardef\colon="303A  % :=

\theoremstyle{plain}

\newtheorem{thm}{Theorem}[section]
\newtheorem{cor}[thm]{Corollary}
\newtheorem{lem}[thm]{Lemma}
\newtheorem{prop}[thm]{Proposition}

\theoremstyle{definition}

\newtheorem{rem}[thm]{Remark}

\newtheorem{exa}[thm]{Example}
\newtheorem{exas}[thm]{Examples}
\newtheorem{defi}[thm]{Definition}

\def\eg{{\em e.g.}\xspace} \def\ie{{\slshape i.e.}\xspace}
\def\cf{cf.~} 
\def\Item[#1]{\item[]\mbox{}\kern-26pt{\bf #1}}

\def\dfn#1{{\bfseries\itshape #1}}

\def\nec{\necmod-stable\xspace}
\def\comonadic{interior\xspace}
\def\acomonadic{an interior\xspace}
\def\Comonadic{Interior\xspace}

\newcommand{\refToSect}[1]{Section~{\ref{sect:#1}}}
\newcommand{\refToDef}[1]{Definition~{\ref{def:#1}}}
\newcommand{\refToDefItem}[2]{\refToDef{#1}{(\ref{def:#1:#2})}}
\newcommand{\refToThm}[1]{Theorem~{\ref{thm:#1}}}
\newcommand{\refToProp}[1]{Proposition~{\ref{prop:#1}}}
\newcommand{\refToPropItem}[2]{condition {(\ref{prop:#1:#2})} in \refToProp{#1}}

\newcommand{\refToCor}[1]{Corollary~{\ref{cor:#1}}}
\newcommand{\refToEx}[1]{Example~{\ref{ex:#1}}}

\def\ple#1{\ensuremath{{({#1})}}}

\def\vuoto{}
\newcommand{\FUN}[4]{\ensuremath{{#2}#1{#3}\rightarrow{#4}}}
\newcommand{\fun}[3]{\relax\def\testa{#1}\relax\ifx\testa\vuoto
  \relax\FUN{}{}{{#2}}{{#3}}\else\relax\FUN{:}{{#1}}{{#2}}{{#3}}\fi}
\newcommand{\dom}{\mathsf{dom}}
\newcommand{\PW}[1]{\ensuremath{\mathop{\mathscr{P}{#1}}}}
\newcommand{\pw}[1]{\relax\def\testa{#1}\relax\ifx\testa\vuoto
  \relax\PW{}\else\relax\PW{\left(#1\right)}\fi}

\def\tt{\ensuremath{\textsf{t\kern-.3ex t}}}
\def\ff{\ensuremath{\textsf{f\kern-.3ex f}}}
\let\Land\wedge

\let\ForalL\forall \def\Forall#1.{\ForalL_{#1}}
\let\ExistS\exists \def\Exists#1.{\ExistS_{#1}}
\def\Lbang{\mathsf{!}} 
\def\Ltensor{\otimes} 
\def\Lunit{1} 
\def\Llimpl{\multimap} 
\def\Lsup{\bigvee} 
\def\Qtl{\mathit{V}} 

\DeclareFontFamily{OT1}{pzc}{}
\DeclareFontShape{OT1}{pzc}{m}{it}{<->s*[1.12]pzcmi7t}{}
\DeclareMathAlphabet{\mathpzc}{OT1}{pzc}{m}{it}
\def\ct#1{\ensuremath{\mathpzc{#1}}}
\def\Ct#1{\ensuremath{\mathbf{#1}}}

\def\id#1{\ensuremath{\mathrm{id}_{#1}}}
\def\Id#1{\ensuremath{\mathrm{Id}_{#1}}}
\def\op{^{\mbox{\normalfont\scriptsize op}}}
\def\blank{\mathchoice{\mbox{--}}{\mbox{--}}
{\mbox{\scriptsize--}}{\mbox{\tiny--}}}

\newcommand{\B}{\ct{B}\xspace}
\newcommand{\C}{\ct{C}\xspace}
\newcommand{\D}{\ct{D}\xspace} 
\newcommand{\K}{\ct{K}\xspace} 
\newcommand{\topos}{\ct{E}\xspace}
\newcommand{\atopos}{\ct{F}\xspace} 
\newcommand{\oneAr}[3]{\ensuremath{{#1}:{#2}\rightarrow{#3}}}
\newcommand{\twoAr}[3]{\ensuremath{{#1}:{#2}\Rightarrow{#3}}}
\def\tdot{\textbf{.}}
\def\lsta{\vrule depth4pt width0pt}
\def\lstb{\vrule height5pt width0pt}
\def\arnta[#1]{\ar[#1]|-*=0[@]{\lsta\tdot}}
\def\arntb[#1]{\ar[#1]|-*=0[@]{\lstb\tdot}}
\newcommand{\TNat}[3]{\ensuremath{{#1}:{#2}
 \stackrel{\makebox{\kern-.3ex\tdot}}\rightarrow{#3}}}

\newcommand{\Op}[1]{\ensuremath{{#1}\op}}
\newcommand{\idobj}[1]{\ensuremath{e_{#1}}}
\newcommand{\idar}[1]{\ensuremath{1_{#1}}}
\newcommand{\Sub}[2][{}]{\relax\def\testa{#2}\relax
  \ensuremath{\mathsf{Sub}_{#1}\relax\ifx\vuoto\testa\relax\else
    \left(#2\right)\fi}}

\def\set#1#2{\ensuremath{\left\{#1\left|\strut\,#2\right.\right\}}}

\newcommand{\Set}{\ct{Set}\xspace}

\newcommand{\Pos}{\ct{Pos}\xspace} 
\newcommand{\One}{\ct{\mathbf{1}}}

\newcommand{\Topop}{\ct{Opn}\xspace} 
\newcommand{\Psh}[1]{\ensuremath{[\Op{#1},\Set]}}
\newcommand{\Fam}[1]{\ensuremath{#1\text{-}\ct{Fam}}}
\newcommand{\IdxPos}{\Ct{Dtn}\xspace} 
\newcommand{\NecIdxPos}{\ensuremath{\necmod\mbox{-}\IdxPos}\xspace}

\def\enab#1{\ensuremath{#1}\xspace}
\newcommand{\cmd}{\enab{\mathsf{K}}}
\newcommand{\acmd}{\enab{\mathsf{J}}}
\newcommand{\cmdfun}{\enab{K}}
\newcommand{\acmdfun}{\enab{J}}
\newcommand{\cmdlift}{\enab{\kappa}}
\newcommand{\acmdlift}{\enab{\psi}}
\newcommand{\counit}{\enab{\epsilon}}
\newcommand{\comul}{\enab{\mu}}
\newcommand{\cmu}{\enab{\mu}}
\newcommand{\cun}{\enab{\nu}}
\edef\cdrestoreat{%%
\noexpand\catcode\lq\noexpand\@=\the\catcode\lq\@}\catcode\lq\@=11
\def\unita{\@ifnextchar[{\UNita}{\UNita[\cmd]}}
\def\UNita[#1]{\enab{\varsigma^{#1}}}
\def\Forget{\@ifnextchar[{\FOrget}{\FOrget[\cmd]}}
\def\ForgetLift{\@ifnextchar[{\FOrgetlift}{\FOrgetlift[\cmd]}}
\def\FOrget[#1]{\ensuremath{\mathit{U}^{#1}}}
\def\FOrgetlift[#1]{\ensuremath{\iota^{#1}}}
\cdrestoreat

\newcommand{\Cmd}[1]{\ensuremath{\mathsf{Cmd}(#1)}}

\newcommand{\unit}{\enab{\eta}}

\newcommand{\LAdj}{\ensuremath{L}\xspace}
\newcommand{\LAdjLift}{\ensuremath{\lambda}\xspace}
\newcommand{\RAdj}{\ensuremath{R}\xspace}
\newcommand{\RAdjLift}{\ensuremath{\rho}\xspace}
\newcommand{\adj}{\ensuremath{\mathbb{A}}\xspace}
\newcommand{\aadj}{\ensuremath{{\mathbb{B}}}\xspace}

\newcommand{\Adj}[1]{\ensuremath{\mathsf{Adj}(#1)}}

\newcommand{\Doc}{\ensuremath{\mathit{P}}\xspace}
\newcommand{\aDoc}{\ensuremath{\mathit{Q}}\xspace}
\newcommand{\bDoc}{\ensuremath{\mathit{M}}\xspace}

\newcommand{\DReIdx}[2]{\ensuremath{{#1}{#2}}}
\newcommand{\order}{\le}

\newcommand{\EM}[2]{\ensuremath{{{#1}^{#2}}}}

\newcommand{\necmod}{\ensuremath{\Box}\xspace}

\def\shopr#1{\ensuremath{\mathrm{#1}}\xspace}
\newcommand{\AMFun}{\shopr{AM}}
\newcommand{\MAFun}{\shopr{MA}}
\newcommand{\CMFun}{\shopr{CM}}
\newcommand{\MCFun}{\shopr{MC}}
\newcommand{\EMFun}{\shopr{EM}}
\newcommand{\EMAFun}{\shopr{EMA}}
\newcommand{\IncFun}{\shopr{Inc}}
\newcommand{\CmdFun}{\shopr{Cmd}}

\def\shbopr#1{\ensuremath{\mathrm{#1}}\xspace}
\newcommand{\TopInt}{\shbopr{int}}
\newcommand{\IntOp}{\shbopr{j}}
\newcommand{\KFr}{\ensuremath{K}\xspace}
\begin{document}

\title{Doctrines, modalities and comonads}

\author{Francesco Dagnino\thanks{DIBRIS, Universit\`a di Genova,
email: francesco.dagnino@dibris.unige.it},
Giuseppe Rosolini\thanks{DIMA, Universit\`a di Genova,
  email: rosolini@unige.it}}
\date{}

\maketitle
\begin{abstract}
Doctrines are categorical structures very apt to study logics of different nature within
a unified environment: the 2-category \IdxPos of doctrines.
Modal \comonadic operators are characterised as particular adjoints in the 2-category
\IdxPos. We show that they can be constructed from comonads in \IdxPos as well as from
adjunctions in it, and the two constructions compare. Finally we show the amount of
information lost in the passage from a comonad, or from an adjunction, to the modal
\comonadic operator.

The basis for the present work is provided by some seminal work of John Power.
\end{abstract}

\section{Introduction}\label{sect:introduction}
The approach to logic proposed by F.W.~Lawvere via hyperdoctrines has proved very fruitful
as it provides an extremely suitable environment where to analyse both syntacic aspects of
logic and semantic aspects as well as compare one with the other, see
\cite{LawvereF:adjif,LawvereF:equhcs}. The suggestion is to see a logic as a functor
\fun{\Doc}{\C\op}{\Pos} from the opposite of a category to the category of posets and
monotone functions where the category \C collects the ``types'' of the logic and terms in
context, a poset $\Doc(c)$ presents the ``properties'' of the type $c$ with the order
relation describing their ``entailments''.
The reader is referred to \refToSect{operanddoc} for the precise details, 
but may just keep in mind, for the present discussion, that the
contravariant powerset functor \fun{\pw{}}{\ct{Set}\op}{\Pos} is an
instance of a doctrine.

One of the main points of Lawvere's structural approach to logic is that all the logical
operators are obtained from adjunctions. That view in itself is very powerful and
contributes to unifying many different aspects in logic. In the present paper, we show
that also a wide class of modal operators, namely, those satisfying axioms T and 4 
as in \refToDef{modality}, is obtained from adjunctions.

Typically, modalities are unary logical operators, which are quite well-understood 
in the context of propositional logic.  
However, their meaning is less clear 
in a typed logical formalism.  
In this setting, there are various semantics which are interrelated, 
and we show that many of these are instances of the general situation of an
adjunction between two homomorphisms of doctrines.

Since they are structured categories, doctrines get swiftly organised in a
2-category. And, as we learned also from the works of John Power, in a 2-category one can
develop a very productive theory of monads and comonads, extending the elementary case of
the 2-category \ct{Cat} of small categories, functors and natural transfomations.

Doctrines are a rather simple categorical framework for logic, but still capable to cover
a large range of examples. We could have considered more general settings such as indexed
preorders (equivalently, faithful fibrations) or even arbitrary fibrations, but we
preferred to keep things at a very simple level as already there one finds many
interesting examples. Yet, after this first step our plan is to extend results to
general fibrations in future work.

We show that an adjunction in the 2-category of doctrines gives rise to a doctrine with a
modal operator.
An adjunction between doctrines is very much like an adjunction between categories: roughly, 
it consists of two doctrines \fun{\Doc}{\C\op}{\Pos} and \fun{\aDoc}{\D\op}{\Pos} and two homomorphisms of doctrines connecting them, 
which should be thought of as an interpretation of \Doc in \aDoc (the left adjoint) and an interpretation of \aDoc in \Doc (the right adjoint). 
Such a situation can be summarised by a modal logic which uses the logic \aDoc to describe properties of  types in \C (the base category of \Doc) 
and the modal operator to recover (an image of) properties described by \Doc. 
In a sense, we extend the logic \Doc through the adjunction to a richer logic and use a modal operator to keep memory of the original logic. 
As we said, many standard approaches to the semantics of modal logic are instances of such construction. 

Taking a slightly different perspective, 
we show that also a comonad in the 2-category of doctrines determines a doctrine with a modal operator, this time  on
the category of coalgebras for the comonad. 
Intuitively, we get a logic where types have a dynamics, given by the coalgebra structure,  
and the modal operator specifies when a property is invariant for such dynamics.

These two constructions are tightly related. 
Relying on results in
\cite{BlackwellR:twodmt}, we show that every comonad in the 2-category of doctrines determines an adjunction, 
hence, also a modal operator. 
In fact, the construction starting from comonads is defined in this way. 
On the other hand, every adjunction determins a comonad, hence a modal operator. 
However, the two construction starting from an adjunction do not coincide, but we show they can be canonically 
compared by a homomorphism of doctrines preserving the modal operator. 

We further our analysis measuring in a categorical form how the
passage to a modal operator hides part of the structure that generated it.

In \refToSect{operanddoc}
we introduce \emph{interior operators} on doctrines, which are the class of modal operators we are interested in. 
In \refToSect{background}
we recall basic notions about comonads and adjunctions in a general 2-category. 
In \refToSect{twocatdoc}
we define the 2-categories of doctrines and doctrines with interior operators that are at the core of our analysis. 
In \refToSect{adjoint}
we show how to construct an interior operator starting from an adjunction between doctrines, while 
in \refToSect{comonad}
we describe the analogous construction starting from a comonad on a doctrine. 
Finally, in \refToSect{picture}
we compare the two constructions showing they are part of local adjunctions, in the sense
of \cite{BettiP88}, between the 2-category of doctrines with modal operator and,
respectively, the 2-category of adjunctions and that of comonads in the 2-category of
doctrines. 
In Appendix~\ref{append}
we sketch an example on how to use our construction to obtain models
of the bang modality of linear logic.

\section{\Comonadic operators and doctrines}\label{sect:operanddoc}

A simple semantic approach to propositional standard modal logic (satisfying axioms T and 4)
would consider an \dfn{\comonadic operator} on a poset \ple{H,\order}, \ie a monotone
function \fun{\IntOp}{H}{H} such that, for all $x \in H$, 
$\IntOp(x) \order x$ and $\IntOp(x) \order \IntOp(\IntOp(x))$, see
\eg\cite{EsakiaL:intlmvt}. The intuition is that the elements of the poset are an
interpretation of (some kind of) formulas, the order relation realizes the entailment
between them, and the \comonadic operator \fun{\IntOp}{H}{H} acts as a modality on
formulas.

From a similar semantic point of view, one could consider a many-sorted logic to be a
\dfn{doctrine} \fun{\Doc}{\C\op}{\Pos}, \ie a (contra)variant functor from a category
\C to the category \Pos of posets and monotone functions. Such a functor is often called an
\dfn{indexed poset} in consonancy with the more general notion of indexed category.

The intuition for a doctrine is that the objects of the category provide the interpretations
of the sorts in the logic and the arrows interpret terms between sorts. For an
object $X$ in \C, the poset $\Doc X$ gives the interpretations for the formulas expressing
the properties of ``arbitrary elements'' of $X$---although no set-theoretic determination of
$X$ may have been provided, see \cite{LawvereF:adjif,LawvereF:equhcs}, but also
\cite{JacobsB:catltt,MaiettiME:eleqc}.

Conjoining these two semantic approaches it is quite natural to consider \comonadic
operators on a doctrine as an extension to many-sorted logic, of the propositional
modal logic satisfying axioms T and 4, like the \necmod-modality, a.k.a. \emph{necessity}
modality, of S4 modal logic.

\begin{defi}\label{def:modality}
Let \fun{\Doc}{\C\op}{\Pos} be a doctrine.
An \dfn{\comonadic} modal operator on \Doc is a natural
transformation \TNat{\necmod}{\Doc}{\Doc} such that,  
for each object $X$ in \C, the following inequalities hold:
\begin{enumerate}[(i)]
\item\label{def:modality:1} $\necmod_X \order_X \id{\Doc X}$
\item\label{def:modality:2} $\necmod_X \order_X \necmod_X \circ \necmod_X$
\end{enumerate}
\end{defi}

Note that standard axioms of the S4 modal operator, see \eg\cite{AwodeyKK14}, require
further structure. But here we consider the 
very simple structure of a poset on the fibres because we want to
focus mainly on the comonadic structure of the modality.

In the following, an element $\alpha\in\Doc X$ of the form
$\alpha=\necmod_X\beta$ for some $\beta\in\Doc X$  will be called
\dfn{\nec}.
An immediate consequence of \refToDef{modality}, obtained combining
the two requirements on $\necmod$, is that
$\necmod_X=\necmod_X\circ\necmod_X$.
Hence \nec elements are the fixed points of
$\necmod_X$, that is, those elements $\alpha\in\Doc X$ such that
$\necmod_X\alpha=\alpha$.  

\begin{exas}\label{ex:into}
Let \fun{\IntOp}{H}{H} be an \comonadic operator on the poset \ple{H,\order}, \ie a monotone
function such that, for all $x \in H$, 
$\IntOp(x) \order x$ and $\IntOp(x) \order \IntOp(\IntOp(x))$.
Given this, we can consider two examples of doctrines with \acomonadic operator:

\begin{enumerate}[(a)] 
\item Let \fun{\hat{H}}{\One\op}{\Pos} be the functor defined on the
category with a single object $\star$ and a single arrow $\id{\star}$
as $\hat{H}(\star) = H$. Then $\IntOp$ is \acomonadic operator
on $\hat{H}$.
\item The functor \fun{H^{(\blank)}}{\Set\op}{\Pos}, 
which maps a set $X$ to $H^X$ with the pointwise order and a function
\fun{t}{X}{Y} to the monotone function \fun{\blank\circ t}{H^Y}{H^X}, 
is a doctrine. The natural transformation
\TNat{\IntOp\circ\blank}{H^{(\blank)}}{H^{(\blank)}} given by
postcomposition with $\IntOp$ is an \acomonadic operator on
$H^{(\blank)}$.
\end{enumerate}

Note that the example in (a) is obtained from that in (b) by precomposing the doctrine
\fun{H^{(\blank)}}{\Set\op}{\Pos} with the (opposite of the) functor
\fun{\star\mapsto\{0\}}{\One}{\Set} which maps the one object $\star$ to a(ny) singleton
set.
\end{exas}

\begin{exa}\label{ex:topol}
Consider the category $\Topop$ of topological spaces and open
continuous maps. Define $\fun{\Doc}{\Topop\op}{\Pos}$ as
$\Doc\ple{X,\tau}=\pw{X}$, the powerset of the set $X$, and
$\DReIdx{\Doc}{t}=t^{-1}$, the inverse image along the open
continuous function \fun{t}{\ple{X,\tau}}{\ple{Y,\sigma}}
Let \ple{X,\tau} be a topological space, then $\tau$ is the set of
fixed points of the \comonadic operator
$\fun{\TopInt_\tau}{\pw{X}}{\pw{X}}$, which maps a subset $A\subseteq X$
to its topological interior.
Since $\TopInt_\tau(A)\subseteq A$ and
$\TopInt_\tau(A)\subseteq\TopInt_\tau(\TopInt_\tau(A))$, for each
$A\subseteq X$, to get an \acomonadic
operator on $\Doc$ we need to prove that $\TopInt_\tau$ is natural.  
Indeed, consider an open continuous map
$\fun{t}{\ple{X,\tau}}{\ple{Y,\sigma}}$, and a subset 
$B\subseteq Y$. So
$t^{-1}(\TopInt_\sigma(B))\subseteq\TopInt_\tau(t^{-1}(B))$ by
continuity of $t$. But also $t(\TopInt_\tau(t^{-1}(B)))\subseteq\TopInt_\sigma(B)$
since the set $t(\TopInt_\tau(t^{-1}(B)))\subseteq B$ is open by openness of $t$. So
$t^{-1}(\TopInt_\sigma(B))=\TopInt_\tau(t^{-1}(B))$ which proves that \TNat{\TopInt}{P}{P}.
\end{exa}

\begin{exa}\label{ex:kripke} 
A \dfn{Kripke frame} is a pair $\KFr = \ple{W,R}$ where $W$ is the set
of \dfn{possible worlds} and $R\subseteq W\times W$ is the
\dfn{accesibility relation}.
On the poset $\pw{W}$ ordered by set inclusion, consider the
monotone function $\fun{\IntOp_R}{\pw{W}}{\pw{W}}$ defined as
\[ \IntOp_R(A) = \set{w \in W}{R(w) \subseteq A} \]
where $R(w) = \set{v \in W}{\ple{w,v} \in R}$.
When $R$ is reflexive and transitive (\ie a preorder on $W$), for any
$w\in W$, we have $w\in R(w)=R(R(w))$. Hence $\IntOp_R$ is \acomonadic operator.
\begin{enumerate}[(a)] 
\item As a particular instance of \refToEx{into}(b), postcomposition with the
\comonadic operator 
\TNat{\IntOp_R\circ\blank}{\pw{W}^{(\blank)}}{\pw{W}^{(\blank)}} 
endows the doctrine $\fun{\pw{W}^{(\blank)}}{\Set\op}{\Pos}$ with 
an \acomonadic operator.
Intuitively, given a ``formula'' $\alpha\in\pw{W}^D$, for an element
$x$ of $D$, the set $\alpha(x)\subseteq W$ consists of those worlds
where $x$ satisfies $\alpha$. Indeed, one can see the data consisting
of the Kripke frame \KFr and the set $D$ as a constant domain skeleton 
as in Definition~1 in \cite{BraunerG07}, where the fibres
$\pw{W}^{D^n}$ enlist all possible interpretations for predicates as $n$ varies.
\item Another doctrine with \acomonadic operator
built from a Kripke frame $\KFr$ with a reflexive and transitive
accessibility relation can be obtained via $W$-indexed families. 
Consider the category \Fam{W} whose
\begin{description}
\Item[objects] are $W$-indexed families of sets, that is, pairs
$X=\ple{\overline{X},(X_w)_{w\in W}}$, where
$X_w\subseteq\overline{X}$, for all $w\in W$, and where
\Item [an arrow] \fun tXY is a function
\fun{t}{\overline{X}}{\overline{Y}} such that, 
for each $w \in W$, $X_w \subseteq t^{-1}(Y_w)$.
\end{description}
Consider the subobject functor
$\fun{\Sub[\Fam{W}]{}}{\Fam{W}\op}{\Pos}$ mapping a $W$-indexed
family to the poset 
$\Sub[\Fam{W}]{X}$ of its subfamilies, \ie a family $A$
such that $\overline{A}\subseteq\overline{X}$ and $A_w \subseteq X_w$
for each $w \in W$, ordered by pointwise inclusion.  
The action on arrows is defined pointwise by inverse image. 
For each $W$-indexed family $X$ the function
$\fun{\necmod_X}{\Sub[\Fam{W}]{X}}{\Sub[\Fam{W}]{X}}$
\[ \left(\necmod_X A\right)_w = \bigcap_{v \in R(w)} A_v \]
is clearly monotone; and it satisfies conditions (\ref{def:modality:1}) and
(\ref{def:modality:2}) in \refToDef{modality} for the same reason as in
the previous example. Moreover, it is natural in $X$ since, for each
function $\fun{t}{Y}{X}$, we have  
\[t^{-1}\big(\left(\necmod_X A\right)_w\big)
=t^{-1} \bigg(\bigcap_{v \in R(w)}A_v\bigg) 
=\bigcap_{v \in R(w)}t^{-1}(A_v)=(\necmod_Yt^{-1}(A))_w\]
for any $w\in W$. Though surprising, we shall see in \refToEx{sndex} that this example is a
universal completion of the previous one in (a).

Intuitively, given a $W$-indexed family $D$, for each $w\in W$, the subset $D_w$ consists
of those elements of $\overline{D}$ which are present at the world $w$, and, given a
``formula'' $\alpha\in \Sub[\Fam W]{D}$, for each world $w\in W$, the set $\alpha_w$
consist of those elements $x$ which are present and satisfy $\alpha$ at $w$.
Indeed, one can see the data consisting
of the Kripke frame \KFr and the $w$-indexed family $D$ as a varying domain skeleton 
as in Definition~7 in \cite{BraunerG07}, with few additional requirements, 
where the fibres $\Sub[\Fam{W}]{D^n}$ enlist all possible interpretations for predicates
as $n$ varies.
\item Yet another possibility is to consider a doctrine
over the category of presheaves on the preorder $\KFr$; we shall
discuss this in \refToEx{presh}, as a particular case of a more
general construction.  
\end{enumerate} 
\end{exa}

\section{Adjunctions and comonads in a 2-category}\label{sect:background}

In this section we recall basic notions which can be introduced in an arbitrary 2-category
with the purpose to use them in the particular case of the 2-category of doctrines.

Given a (strict) 2-category \K, we denote 0-cells as $A$, $B$, $C,\ldots$, which we shall
refer to also as \dfn{objects} of \K; a 1-cell, also referred to as \dfn{1-arrow}, from $A$
to $B$ will be written as $\oneAr{a}{A}{B}$ while a 2-cell, or \dfn{2-arrow}, from the
1-cell $a$ to the 1-cell $b$ will be written as $\twoAr{\alpha}{a}{b}$. Composition of
1-cells and horizontal composition of 2-cells is denoted as $\circ$, and often omitted---we
shall use it mainly to emphasise the composition of functions and functors. 
The identity 1-cell on the object $A$ is denoted by \idobj{A} and 
the identity 2-cell on the 1-cell $a$ is denoted by \idar{a}.
A horizontal composition with a 2-identity cell \idar{a} will be written simply as $\alpha
a$. Vertical composition of 2-cells is denoted as $\cdot$. So, for instance, the defining
property of vertical composition of natural transformations would be written as something
like $(\psi\cdot\phi)_C=\psi_C\circ\phi_C$.

Many well-known concepts from standard category theory can be transferred to an arbitrary
2-category $\K$; a basic reference is \cite{Street72}.
\begin{defi}\label{def:adjcom}
Let \K be a 2-category. 
\begin{enumerate}[(i)]
\item\label{def:adjcom:1}
An \dfn{adjunction} \adj in $\K$ consists of the following data: two
objects $C$ and $D$, 
two 1-arrows $\oneAr{l}{C}{D}$ and $\oneAr{r}{D}{C}$, and
two 2-arrows $\twoAr{\unit}{\idobj{C}}{rl}$ and
$\twoAr{\counit}{lr}{\idobj{D}}$,  
such that the following triangles of 2-arrows commute
\begin{equation}\label{triai}
\xymatrix@C=3em@R=2em{l\ar@{=>}[r]^-{l\unit}\ar@{=>}[rd]_-{\idar{l}}&
lrl\ar@{=>}[d]^-{\counit l}\\&l}\qquad
\xymatrix@C=3em@R=2em{r\ar@{=>}[r]^-{\unit r}\ar@{=>}[rd]_-{\idar{r}}&
rlr\ar@{=>}[d]^-{r\counit }\\&r.}
\end{equation}
\item\label{def:adjcom:2}
A \dfn{comonad} $\mathbb{c}$ in \K consists of an object $A$, a 1-arrow
$\oneAr{c}{A}{A}$, and two 2-arrows $\twoAr{\cun}{c}{\idobj{A}}$ and
$\twoAr{\cmu}{c}{cc}$, such that the following diagrams of 2-arrows
commute 
\begin{equation}\label{dcom}
\xymatrix@C=3em@R=2em{
&c\ar@{=>}[ld]_-{\idar{c}}\ar@{=>}[rd]^-{\idar{c}}\ar@{=>}[d]^-{\cmu}\\
c&cc\ar@{=>}[l]^-{c\cun}\ar@{=>}[r]_-{\cun c}&c}\qquad
\xymatrix@C=3em@R=2em{c\ar@{=>}[r]^-{\cmu}\ar@{=>}[d]_-{\cmu}&
cc\ar@{=>}[d]^-{c\cmu}\\cc\ar@{=>}[r]^-{\cmu c}&ccc.}
\end{equation}
\item\label{def:adjcom:3} In line with \cite{Street72,PowerW02}, one says that
\K{} \dfn{admits the Eilenberg-Moore construction} for the comonad \ple{A,c,\cmu,\cun}
if there is a universal representation of the following 2-problem: given an object $B$ in \K,
objects are pairs \ple{x,\xi} with
\begin{equation}\label{udata}
\vcenter{\xymatrix@C=4em@R=1ex{&A\ar[dd]^-{c}_(.55){}="t"\\
B\ar[ru]^-{x}\ar[rd]_-{x}^(.68){}="s"\\
\ar@{=>}"s";"t"^-{\xi}&A}}
\end{equation}
and such that the diagrams of 2-arrows 
\begin{equation}\label{1cohe}
\vcenter{\xymatrix{x\ar@{=>}[r]^-{\xi}\ar@{=>}[d]_-{\xi}&
cx\ar@{=>}[d]^-{\cmu x}\\
cx\ar@{=>}[r]_-{c\xi}&ccx}}\qquad
\vcenter{\xymatrix{x\ar@{=>}[r]^-{\xi}\ar@{=>}[rd]_-{\idar{x}}&
cx\ar@{=>}[d]^-{\cun x}\\&x}}
\end{equation}
commute; an arrow \oneAr{\gamma}{\ple{x,\xi}}{\ple{y,\zeta}} is a
2-arrow \twoAr{\gamma}{x}{y} such that the following diagram commutes
\begin{equation}\label{2cohe}
\vcenter{\xymatrix{x\ar@{=>}[d]_{\gamma}\ar@{=>}[r]^-{\xi}&
cx\ar@{=>}[d]^-{c\gamma}\\
y\ar@{=>}[r]^-{\zeta}&cy.}}
\end{equation}
\end{enumerate}
\end{defi}
Spelling out the data for an Eilenberg-Moore construction for the comonad
$\mathbb{c}=\ple{A,c,\cmu,\cun}$, it requires that there is an object
\EM{A}{\mathbb{c}} in \K together with a 1-arrow and a 2-arrow as in
\[\xymatrix@C=4em@R=1ex{&A\ar[dd]^-{c}_(.55){}="t"\\
\EM{A}{\mathbb{c}}
\ar[ru]^-{u^{\mathbb{c}}}
\ar[rd]_-{u^{\mathbb{c}}}^(.68){}="s"\\
\ar@{=>}"s";"t"^-{\omega^{\mathbb{c}}}&A}\]
which satisfy the commutative diagrams in (\ref{1cohe}). Moreover, for any
object $B$ in \K, every pair \ple{x,\xi} as in (\ref{udata}) satisfying (\ref{1cohe}) can
be obtained by precomposition
\[\vcenter{\xymatrix@C=4em@R=1ex{&A\ar[dd]^-{c}_(.55){}="t"\\
B\ar[ru]^-{x}\ar[rd]_-{x}^(.68){}="s"\\
\ar@{=>}"s";"t"^-{\xi}&A}}=
\vcenter{\xymatrix@C=1.5em@R=1ex{&&&A\ar[dd]^-{c}_(.55){}="t"\\
B\ar[r]^-{x'}&\EM{A}{\mathbb{c}}
\ar[rru]^-{u^{\mathbb{c}}}
\ar[rrd]_-{u^{\mathbb{c}}}^(.68){}="s"\\
&&\ar@{=>}"s";"t"^-{\omega^{\mathbb{c}}}&A}}\]
for a unique 1-arrow \oneAr{x'}{B}{\EM{A}{\mathbb{c}}}, and similarly for arrows
\oneAr{\gamma}{\ple{x,\xi}}{\ple{y,\zeta}} between pairs:
\[\vcenter{\xymatrix@C=4em@R=1ex{B\ar@<1ex>@/^/[r]^-x_{}="s"
\ar@<-1ex>@/_/[r]_-y^{}="t"&A\ar@{=>}"s";"t"^-{\gamma}}}=
\vcenter{\xymatrix@C=1.5em@R=1ex{B\ar@<1ex>@/^/[rr]^-{x'}_{}="s"
\ar@<-1ex>@/_/[rr]_-{y'}^{}="t"&&
\EM{A}{\mathbb{c}}\ar@{=>}"s";"t"^-{\gamma'}\ar[r]^-{u^{\mathbb{c}}}&A}}\]
for a unique 2-arrow \twoAr{\gamma'}{x'}{y'} in \K.

In case the universality condition is verified for each comonad in \K, it can be restated
in terms of a 2-adjunction after introducing the appropriate\footnote{There are many
reasonable 2-categories whose objects are adjunctions in \K. In this paper, the 2-category
\Adj{\K} we introduce is the one that gives rise to the 2-adjunction with \Cmd{\K}.}
2-category \Adj{\K} of adjunctions in \K and the 2-category \Cmd{\K} of comonads in
\K. Since we can safely refer the reader to \cite{PowerW02} for a very clear presentation
of the general setup, we limit ouselves to recapping the main diagram of 2-adjunctions:
\begin{equation}\label{maind}
\vcenter{\xymatrix{
\K \ar@/^8pt/[rr]^{\IncFun} \ar@{}[rr]|{\bot}  &&
\Cmd{\K} \ar@/^8pt/[ll]^{\EMFun} \ar@/_8pt/[rr]_{\EMAFun} \ar@{}[rr]|{\bot} &&
\Adj{\K} \ar@/_8pt/[ll]_{\CmdFun}}}
\end{equation}
where the 2-functor $\IncFun$ sends an object $A$ in \K to the identity comonad
\ple{A,\idobj{A},\idar{\idobj{A}},\idar{\idobj{A}}} on $A$, and the 2-functor $\EMFun$
sends a comonad $\mathbb{c}=\ple{A,c,\cmu,\cun}$ to its Eilenberg-Moore object
\EM{A}{\mathbb{c}}; while the 2-functor $\CmdFun$ sends an adjunction
$\adj=\ple{C,D,l,r,\unit,\counit}$ to the associated comonad \ple{D,lr,l\unit r,\counit},
and the 2-functor $\EMAFun$ sends a comonad $\mathbb{c}$ to the Eilenberg-Moore adjunction
between $A$ and \EM{A}{\mathbb{c}}.

\begin{exa}
Although the terminology already suggests clearly the kind of generalization adopted, we
hasten to point out that in the 2-category \ct{Cat} of (small) categories, functors and
natural transfomations, the definitions in (i) and (ii) instantiate exactly to the usual
notions of (standard) adjunction between categories $l\dashv r$---where $\eta$ and
$\epsilon$ are the unit and the counit of the adjunction---, and to comonads. Clearly,
\ct{Cat} admits the Eilenberg-Moore construction for every comonad.
\end{exa}

In the next sections we shall characterize adjunctions and comonads in the 2-category
\IdxPos of doctrines.

\section{The 2-category of doctrines}\label{sect:twocatdoc}

The 2-category \IdxPos of \dfn{doctrines}
consists of the following data:
\begin{description}
\Item [objects] are doctrines, \ie a functor $\fun{\Doc}{\C\op}{\Pos}$ from the opposite
of a category \C to the category \Pos of posets and monotone functions---in the
nomenclature of indexed categories, the category \C is named the \dfn{base} of the
doctrine, for $X$ an object in \C the poset $\Doc(X)$ is the \dfn{fibre over} $X$, and for
\fun{t}{X}{Y} an arrow in \C, the monotone function \fun{\Doc t}{\Doc Y}{\Doc X} is called
\dfn{reindexing along} $t$;\footnote{In the following, we may sometime refer to a doctrine
as a pair \ple{\C,\Doc} in order to make the base \C of the doctrine conspicous.}
\Item [a 1-arrow] 
\oneAr{\ple{F,f}}{\Doc}{\aDoc} from the doctrine \fun{\Doc}{\C\op}{\Pos} to the doctrine
\fun{\aDoc}{\D\op}{\Pos} is a pair where the first component $\fun{F}{\C}{\D}$ is a
functor and the second component $\TNat{f}{\Doc}{\aDoc F\op}$ is a natural transformation;
\Item [a 2-arrow] 
\twoAr{\theta}{\ple{F,f}}{\ple{F',f'}}
is a natural transformation $\TNat{\theta}{F}{F'}$ such that, 
for each object $X$ in \C, 
$f_X\order_X(\DReIdx{\aDoc}{\theta\op})_X\circ f'_X$. 
\Item [Composition of 1-arrows]
\oneAr{\ple{G,g}}{\ple{\B,\bDoc}}{\ple{\C,\Doc}} and
\oneAr{\ple{F,f}}{\ple{\C,\Doc}}{\ple{\D,\aDoc}} is (essentially)
pairwise 
\oneAr{\ple{FG,(fG\op)\cdot g}}{\ple{\B,\bDoc}}{\ple{\D,\aDoc}}.
\Item [Composition of 2-arrows]
\twoAr{\theta}{\ple{F,f}}{\ple{F',f'}} and
\twoAr{\zeta}{\ple{F',f'}}{\ple{F'',f''}} is the natural transformation
$\twoAr{(\zeta_X\circ\theta_X)_{X\in\C_0}}{\ple{F,f}}{\ple{F'',f''}}$
since, for any object $X$ in \C,
\[f_X\order_X\DReIdx{\aDoc}{(\theta\op{}_X)}\circ f'_X
\order_X\DReIdx{\aDoc}{(\theta\op{}_X)}\circ\DReIdx{\aDoc}{(\zeta\op{}_X)}\circ f''_X
\order_X\DReIdx{\aDoc}{((\zeta\circ\theta)\op{}_X)}\circ f''_X.\]
\end{description}

There is an obvious forgetful 2-functor \fun{}{\IdxPos}{\Ct{Cat}} to the 2-category of
categories, functors and natural transformations, which maps a doctrine \ple{\C,\Doc} to
its base category \C, and acts similarly on the arrows.
Note that such a 2-functor is actually a 2-fibration, in the sense of \cite{Hermida99}, 
where cartesian 1-arrows are ``chang of base'', that is, arrows of the form \ple{F,\id{}}, 
while vertical 1-arrows ``fibred'', that is, arrows of the form \ple{\Id,f}, which act
only on the fibres.\footnote{Many notions in this paper can be phrased using the language
of 2-fibrations, but with the hope to keep the presentation more accessible, we shall just
highlight the connection in a few important cases.}

We define also the 2-category \NecIdxPos of doctrines endowed with \acomonadic
operator as follows: 
\begin{description}
\Item [objects] are pairs \ple{\Doc,\necmod} where \Doc is a
doctrine and $\necmod$ is \acomonadic operator on \Doc;
\Item [a 1-arrow] from \ple{\Doc,\necmod} to
\ple{\aDoc,\necmod'} is a 1-arrow
$\oneAr{\ple{F,f}}{\Doc}{\aDoc}$ in \IdxPos such that, for each
object $X$ in the base category of $\Doc$, we have
$f_X \circ\necmod_X\order\necmod'_{FX}\circ f_X$;
\Item [a 2-arrow] from \ple{F,f} to \ple{G,g} is a 2-arrow
\twoAr{\theta}{\ple{F,f}}{\ple{G,g}} in \IdxPos. 
\Item[Compositions] are inherited from those of the 2-category \IdxPos.
\end{description}
It is easy to verify that the requirement on the component $f$ of a 1-arrow in \NecIdxPos is
equivalent to the condition that $\necmod'_{FX}\circ f_X\circ\necmod_X=f_X\circ\necmod_X$,
\ie $f_X$ maps \nec elements to \nec elements. 

\begin{exa}\label{ex:fstex}
Consider the forgetful functor \fun{U}{\Topop}{\Set}, and for a
topological space \ple{X,\tau} let
\fun{u_X=\id{\pw{X}}}{\pw{X}}{\pw{X}}. If \ple{P,\TopInt} is as in
\refToEx{topol}, then
\fun{\ple{U,u}}{\ple{P,\TopInt}}{\ple{\pw{},\Id{\pw{}}}} is a 1-arrow
in \NecIdxPos.
\end{exa}
\begin{exa}\label{ex:sndex}
For a Kripke frame $\KFr=\ple{W,R}$ where $R$ is reflexive and transitive, the pairs
\ple{\pw{W}^{(\blank)},\IntOp_R\circ\blank} and \ple{\Sub[\Fam{W}]{},\necmod}, introduced
in \refToEx{kripke}, are objects of \NecIdxPos.

Consider the functor \fun{C}{\Set}{\Fam{W}} which maps a
set $S$ the pair \ple{S,(S)_{w\in W}} where the second component is
the constant family of value $S$. Also, for $\alpha\in\pw{W}^S$,
consider the $W$-indexed family given by
\[\left(c_S(\alpha)\right)_w\colon=\set{s\in S}{w\in\alpha(s)}.\]
Then
\fun{\ple{C,c}}{\ple{\pw{W}^{(\blank)},\IntOp_R\circ\blank}}
 {\ple{\Sub[\Fam{W}]{},\necmod}} is a 1-arrow in \NecIdxPos.

One can show that the 1-arrow \fun{\ple{C,c}}{\pw{W}^{(\blank)}}{\Sub[\Fam{W}]{}} is the
comprehension completion of the doctrine \fun{\pw{W}^{(\blank)}}{\Set\op}{\Pos}, and that
the \comonadic operator \necmod is the canonical extension of the other operator
$\IntOp_R\circ\blank$, see \cite{MaiettiME:quofcm,StreicherT:semott}.
\end{exa}

\begin{rem}\label{rem:forget}
There is a forgetful 2-functor \fun{}{\NecIdxPos}{\IdxPos} which deletes the
\comonadic operator. It has a right 2-adjoint,
which sends a doctrine \fun{\Doc}{\C\op}{\Pos} to
\ple{\Doc,\id{}} and is the identity both on 1-arrows and
2-arrows.  
Indeed, for any object \ple{\Doc,\necmod} in \NecIdxPos the inequality  
$\necmod_X\order\id{\Doc X}$ holds; so for
any 1-arrow $\oneAr{\ple{F,f}}{\Doc}{\aDoc}$ in \IdxPos we have
$f_X\circ\necmod_X\order f_X$ by monotonicity of $f_X$.
\end{rem}

\section{\Comonadic modalities from adjunctions}\label{sect:adjoint}

The main goal of this section is to connect \comonadic operators as in
\refToDef{modality} and adjunctions in \IdxPos.  First we characterise the general
2-categorical notion of adjunction, as introduced in \refToSect{background}, for the
particular case of the 2-category \IdxPos in terms of the functors and natural
transformations involved. 

\begin{prop}\label{prop:idxpos-adj}
An adjunction in the 2-category \IdxPos in the sense of
\refToDefItem{adjcom}{1} is completely determined by an octuple
\ple{\Doc,\aDoc,\LAdj,\LAdjLift,\RAdj,\RAdjLift,\unit,\counit}, where
\fun{\Doc}{\C\op}{\Pos} and \fun{\aDoc}{\D\op}{\Pos} are
doctrines, \fun{\LAdj}{\C}{\D} and \fun{\RAdj}{\D}{\C} are
functors, \TNat{\LAdjLift}{\Doc}{\aDoc\LAdj\op},
\TNat{\RAdjLift}{\aDoc}{\Doc\RAdj\op},
\TNat{\unit}{\Id{\D}}{\RAdj\LAdj} and
\TNat{\counit}{\LAdj\RAdj}{\Id{\D}}
are natural transformations such that
\begin{enumerate}[{\normalfont(i)}]
\item\label{prop:idxpos-adj:1}
\ple{\C,\D,\LAdj,\RAdj,\unit,\counit} is an adjunction in \Ct{Cat};
\item\label{prop:idxpos-adj:2}
\oneAr{\ple{\LAdj,\LAdjLift}}{\Doc}{\aDoc} and
\oneAr{\ple{\RAdj,\RAdjLift}}{\aDoc}{\Doc} are 1-arrows in \IdxPos;
\item\label{prop:idxpos-adj:3}
\twoAr{\unit}{\ple{\Id{\C},\id{\Doc}}}
{\ple{\RAdj\LAdj,(\RAdjLift\LAdj\op)\LAdjLift}} and
\twoAr{\counit}{\ple{\LAdj\RAdj,(\LAdjLift\RAdj\op)\RAdjLift}}
{\ple{\Id{\D},\id{\aDoc}}} are 2-arrows in \IdxPos.
\end{enumerate}
\end{prop}

\begin{proof}
If \ple{P,Q,l,r,\unit,\counit} is an adjunction in
\IdxPos, applying the forgetful functor \fun{}{\IdxPos}{\Ct{Cat}} one
gets immediately \ref{prop:idxpos-adj:1} where $L$ and $R$ are the
first components of $l$ and $r$ respectively. The rest of the proof is
plain bookkeeping.
\end{proof}

As for any 2-category, one can consider the 2-category \Adj{\IdxPos} of adjunctions
in \IdxPos. The following proposition is just as straightforward as the
previous one.

\begin{prop}\label{prop:catadj}
The 2-category \Adj{\IdxPos} of adjunctions in \IdxPos
has objects which are adjunctions
$\adj=\ple{\Doc^\adj,\aDoc^\adj,\LAdj^\adj,\LAdjLift^\adj,
\RAdj^\adj,\RAdjLift^\adj,\unit^\adj,\counit^\adj}$
as in \refToProp{idxpos-adj}, where \fun{\Doc^\adj}{(\C^\adj)\op}{\Pos} and
\fun{\aDoc^\adj}{(\D^\adj)\op}{\Pos}.

A 1-arrow
\oneAr{\ple{F,f,G,g,\theta}}{\adj}{\aadj} 
in \Adj{\IdxPos} consists of two 1-arrows 
\oneAr{\ple{F,f}}{\Doc^\adj}{\Doc^\aadj} and
\oneAr{\ple{G,g}}{\aDoc^\adj}{\aDoc^\aadj},
and a 2-arrow 
\twoAr{\theta}{\ple{F\RAdj^\adj,(f(\RAdj^\adj)\op)\RAdjLift^\adj}}
{\ple{\RAdj^\aadj G,(\RAdjLift^\aadj G\op)g}} in \IdxPos such that the triple
\ple{F,G,\theta} is a homomorphism of adjunctions in \Ct{Cat}, and the two natural
transformations \TNat{(g(\LAdj^\adj)\op)\lambda^\adj}{\Doc^\adj}{\aDoc^\aadj (G\LAdj^\adj)\op} and
\TNat{(\lambda^\aadj F\op)f}{\Doc^\adj}{\aDoc^\aadj(\LAdj^\aadj F)\op} coincide (note that
$G\LAdj^\adj=\LAdj^\aadj F$ by the first condition).
 
A 2-arrow
\twoAr{\ple{\alpha,\beta}}{\ple{F, f, G, g,\theta}}{\ple{F',f',G',g',\theta'}}
in \Adj{\IdxPos} consists of two 2-arrows \twoAr{\alpha}{\ple{F,f}}{\ple{F',f'}} and
\twoAr{\beta}{\ple{G,g}}{\ple{G',g'}} in \IdxPos such that
\ple{\alpha,\beta} is a 2-cell from the adjunction homomorphism \ple{F, G, \theta} to the
adjunction homomorphism \ple{F',G',\theta'} in \Ct{Cat}. 
\end{prop}

\begin{rem}
To elucidate the conditions in \refToProp{catadj} in terms of some diagrams, consider first
that the forgetful 2-functor \fun{}{\IdxPos}{\Ct{Cat}} extends to a 2-functor
\fun{}{\Adj{\IdxPos}}{\Adj{\Ct{Cat}}}. Hence
the condition that the triple \ple{F,G,\theta} is a homomorphism of adjunctions in \Ct{Cat}
requires that the diagram of functors 
\[\xymatrix{\C^\adj\ar[r]^-F\ar[d]_-{\LAdj^\adj}&\C^\aadj\ar[d]^-{\LAdj^\aadj}\\
\D^\adj\ar[r]^-G&\D^\aadj}\]
commutes as well as (either of) the diagrams of natural transformations
\[\xymatrix{F\arnta[d]_-{\unit^\aadj F}\arnta[r]^-{F\unit^\adj}&
F\RAdj^\adj\LAdj^\adj\arntb[d]^-{\theta\LAdj^\adj}\\
\RAdj^\aadj\LAdj^\aadj F\ar@{2-2-}[r]&\RAdj^\aadj G\LAdj^\adj}\kern6em
\xymatrix{\LAdj^\aadj F\RAdj^\adj\ar@{2-2-}[d]\arnta[r]^-{\LAdj^\aadj\theta}&
\LAdj^\aadj\RAdj^\aadj G\arntb[d]^-{\counit^\aadj G}\\
G\LAdj^\adj\RAdj^\adj\arnta[r]_-{G\counit^\adj}&G}\]
as the two commutativity conditions are equivalent. For instance, if we assume the first
commutes, postcomposing it with $\LAdj^\aadj$ and precomposing it with $\RAdj^\adj$,
and using the naturality of $\theta$ and $\counit^\aadj$ and the triangular identities
of adjunctions, we get the second as depicted in the following diagram
\[\xymatrix@R=2.6em@C=4.5em{
\LAdj^\aadj F\RAdj^\adj\arntb[d]^-{\LAdj^\aadj\unit^\aadj F\RAdj^\adj}
\arnta[r]^-{\LAdj^\aadj F\unit^\adj\RAdj^\adj}
\ar@<-10pt>@/_25pt/[dd]|-*=0[@]{\lstb\tdot}_(.4){\id{}}
\ar@<5pt>@/^25pt/[rr]|-*=0[@]{\lsta\tdot}^(.4){\id{}}&
\LAdj^\aadj F\RAdj^\adj\LAdj^\adj\RAdj^\adj
\arntb[d]^-{\LAdj^\aadj\theta\LAdj^\adj\RAdj^\adj}
\arnta[r]^-{\LAdj^\aadj F\RAdj^\adj\counit^\adj}&
\LAdj^\aadj F\RAdj^\adj\arntb[d]^-{\LAdj^\aadj\theta}\\
\LAdj^\aadj\RAdj^\aadj\LAdj^\aadj F\RAdj^\adj\ar@{2-2-}[r]
\arntb[d]^-{\counit^\aadj\LAdj^\aadj F\RAdj^\adj}&
\LAdj^\aadj\RAdj^\aadj G\LAdj^\adj\RAdj^\adj
\arnta[r]^-{\LAdj^\aadj\RAdj^\aadj G\counit^\adj}&
\LAdj^\aadj\RAdj^\aadj G\arntb[d]^-{\counit^\aadj G}\\
\LAdj^\aadj F\RAdj^\adj\ar@{2-2-}[r]&
G\LAdj^\adj\RAdj^\adj\arnta[r]^-{G\counit^\adj}&
\LAdj^\aadj\RAdj^\aadj G}\]

The condition that the pair \ple{\alpha,\beta} is a 2-cell from the adjunction
homomorphism \ple{F,G,\theta} to the adjunction homomorphism \ple{F',G',\theta'} in
\Ct{Cat} translates into commutativity of the following diagrams of natural
transformations
\[\xymatrix{
\LAdj^\aadj F\arnta[r]^-{\LAdj^\aadj\alpha}\ar@{2-2-}[d]&\LAdj^\aadj F'\ar@{2-2-}[d]\\
G\LAdj^\adj\arnta[r]^-{\beta\LAdj^\adj}&G'\LAdj^\adj.} 
\kern6em
\xymatrix{
F\RAdj^\adj\arnta[r]^-{\alpha\RAdj^\adj}\arnta[d]_{\theta}&F'\RAdj^\adj\arntb[d]^{\theta'}\\
\RAdj^\aadj G\arnta[r]^-{\RAdj^\aadj\beta}&\RAdj^\aadj G'}\]
\end{rem}

From now on, when referring to an adjunction in the 2-category
\IdxPos, we shall take advantage of \refToProp{idxpos-adj} and write
it as an octuple
\ple{\Doc,\aDoc,\LAdj,\LAdjLift,\RAdj,\RAdjLift,\unit,\counit}.

\begin{exa}\label{ex:pb-cat}
Examples are many as any adjunction between categories with pullbacks gives rise to an
adjunction between the doctrines of subobjects. In details, given a category with
pullbacks \C, one can define a functor $\fun{\Sub[\C]{}}{\C\op}{\Pos}$
taking advantage of the fact that pulling back preserves monos. The
functor maps an object to the poset of its subobjects and reindexing
along \fun{f}{X'}{X} is as follows:
a subobject $\xymatrix@1{[A\ \ar@<-.2ex>@{^(->}[r]^-{\alpha}&X]}$, determined
by the isomorphism class of the mono $\alpha$, is taken to the subobject
determined by the mono $\alpha'$ obtained as a pullback
\[\xymatrix{A'\ \ar[d]\ar@{^{(}->}[r]^-{\alpha'}
&X'\ar[d]^{f}\\
A\ \ar@{^{(}->}[r]_-{\alpha}&X.}\]
Let \D be also a category with pullbacks, and consider an
adjunction \ple{\C,\D,\LAdj,\RAdj,\unit,\counit} where
\fun{\LAdj}{\C}{\D} preserves pullbacks (as a right adjoint, the
functor \fun{R}{\D}{\C} preserves all existing limits).
Between the doctrines $\fun{\Sub[\C]{}}{\C\op}{\Pos}$ and
$\fun{\Sub[\D]{}}{\D\op}{\Pos}$ 
there are 1-arrows of \IdxPos 
$\oneAr{\ple{\LAdj,\LAdjLift}}{\Sub[\C]{}}{\Sub[\D]{}}$
and
$\oneAr{\ple{\RAdj,\RAdjLift}}{\Sub[\D]{}}{\Sub[\C]{}}$,
where for $X$ in \C and $Y$ in \D
\[\LAdjLift_X(\xymatrix@1{[A\ \ar@{^(->}[r]^-{\alpha}&X]})= 
\xymatrix@1{[\LAdj A\ \ar@{^(->}[r]^-{\LAdj\alpha}& \LAdj X]}
  \qquad
\RAdjLift_X(\xymatrix@1{[B\ \ar@{^(->}[r]^-{\beta}&Y]})= 
\xymatrix@1{[\RAdj A\ \ar@{^(->}[r]^-{\RAdj\beta}& \RAdj Y].}\]
The naturality of $\LAdjLift$ and $\RAdjLift$ follows since reindexing
is given by pulling back, and $\LAdj$ and $\RAdj$ preserve pullbacks.  
To see that
\ple{\Sub[\C]{},\Sub[\D]{},\LAdj,\LAdjLift,\RAdj,\RAdjLift,\unit,\counit}
is an adjunction in \IdxPos there remains to check that 
\twoAr{\unit}{\ple{\Id{\C},\id{\Sub[\C]{}}}}
{\ple{\RAdj\LAdj,(\RAdjLift\LAdj\op)\LAdjLift}} and
\twoAr{\counit}{\ple{\LAdj\RAdj,(\LAdjLift\RAdj\op)\RAdjLift}}
{\ple{\Id{\D},\id{\Sub[\D]{}}}}
are 2-arrows of \IdxPos: in other words,
for any $\xymatrix@1{[A\ \ar@<-.2ex>@{^(->}[r]^-{\alpha}&X]}$ and
$\xymatrix@1{[B\ \ar@<-.2ex>@{^(->}[r]^-{\beta}&Y]}$, we have 
\[[\alpha]\order\Sub[\C]{\unit_X}[\RAdj\LAdj(\alpha)]
\quad\mbox{and}\quad
[\LAdj\RAdj(\beta)]\order\Sub[\D]{\counit_Y}[\beta].\]
But this follows from naturality of \unit and \counit together with the
reindexing pullbacks
\[\xymatrix@C=3em{A\ \ar@{^{(}->}@/^5pt/[drr]^{\alpha}
\ar@/_5pt/[ddr]_{\unit_A}\ar@{-->}[dr]\\
&P\ \ar[d]\ar@{^{(}->}[r]\ar@{}[dr]|{\text{p.b.}}&X\ar[d]^{\unit_X}\\
&\RAdj\LAdj A\ \ar@{^{(}->}[r]_{\RAdj\LAdj(\alpha)}&\RAdj\LAdj X}
\qquad\xymatrix@C=3em{\LAdj\RAdj B\ \ar@{-->}[dr]
\ar@{^{(}->}@/^5pt/[drr]^{\LAdj\RAdj(\beta)}\ar@/_5pt/[ddr]_{\counit_B}&& \\
&P\ \ar[d]\ar@{^{(}->}[r]\ar@{}[dr]|{\text{p.b.}}&
\LAdj\RAdj Y\ar[d]^{\counit_Y}\\ 
&B\ \ar@{^{(}->}[r]_{\beta}&Y}\]
\end{exa}

We now put to use the characterisation in \refToProp{idxpos-adj} to construct \acomonadic 
operator starting from an adjunction of doctrines.
We begin the process performing the construction for a very specific type of adjunctions: adjunctions between vertical 1-arrows.

\begin{prop}\label{prop:adj-to-modality}
Let $\fun{\Doc}{\C\op}{\Pos}$ and $\fun{\aDoc}{\C\op}{\Pos}$ be
doctrines, and suppose the octuple
\ple{\Doc,\aDoc,\Id{\C},\LAdjLift,
\Id{\C},\RAdjLift,\id{\Id{\C}},\id{\Id{\C}}}
is an adjunction in \IdxPos. Then
\begin{enumerate}[{\normalfont(i)}]
\item for each object $X$ in \C, the
following adjunction holds between the fibres
\[\xymatrix{\Doc X \ar@/^8pt/[rr]^{\LAdjLift_X}  \ar@{}[rr]|{\bot} && 
\ar@/^8pt/[ll]^{\RAdjLift_X} \aDoc X,}\quad X\in\C_0;\]
\item $\necmod = \LAdjLift \cdot \RAdjLift$ is \acomonadic
operator on $\aDoc$. 
\end{enumerate}
\end{prop}

\begin{proof}
By \refToProp{idxpos-adj}, the hypothesis ensures that
\twoAr{\id{\Id{\C}}}{\ple{\Id{\C},\id{}}}{\ple{\Id{\C},{\RAdjLift\LAdjLift}}}
and
\twoAr{\id{\Id{\C}}}{\ple{\Id{\C},\LAdjLift\RAdjLift}}{\ple{\Id{\C},\id{}}}
are 2-arrows in \IdxPos. From this, the conclusion follows directly.
\end{proof}

\begin{exa}
Recall from \cite{RosenthalK:quaata} that a \emph{commutative quantale} is a complete
lattice endowed with further structure 
\ple{\Qtl,\Lsup,\order,\Ltensor,\Lunit} where \ple{\Qtl,\Lsup,\order} is a complete
lattice, \ple{\Qtl,\Ltensor,\Lunit} is a commutative monoid such that the operation
$\Ltensor$ distributes over sups:
\[x\Ltensor\left(\Lsup\kern-.8ex\strut_{i\in I}\,x_i\right)=
\Lsup\kern-.8ex\strut_{i\in I}(x\Ltensor x_i)\]
for elements $x$ and families $(x_i)_{i\in I}$ in $\Qtl$---note that this yields that
$\Ltensor$ is monotone in its two arguments.

Let
$R_\Qtl = \{x\in\Qtl\mid x\order \Lunit\mbox{ and } x\order x\Ltensor x\} \subseteq \Qtl$. 
It is easy to check that $\Lunit\in R_\Qtl$ and $R_\Qtl$ is closed with respect to
$\Ltensor$ and $\Lsup$. Hence \ple{R_\Qtl,\Lsup,\order,\Ltensor,\Lunit} is a commutative
quantale. Let $\fun{\iota}{R_\Qtl}{\Qtl}$ be the inclusion function which clearly
preserves sups. Its right adjoint $\fun{r}{\Qtl}{R_\Qtl}$ is determined as
$r(x) = \Lsup \{ y \in R_\Qtl \mid  y\order x\}$.

Consider the doctrine $\fun{\Qtl^{(\blank)}}{\Set\op}{\Pos}$ and
$\fun{R_\Qtl^{(\blank)}}{\Set\op}{\Pos}$ mapping a set $X$ to the sets of functions
$\Qtl^X$ and $R_\Qtl^X$, ordered pointwise, and acting on functions by precomposition. 
And the 1-arrow
$\oneAr{\ple{\Id{},\iota\circ\blank}}{R_\Qtl^{(\blank)}}{\Qtl^{(\blank)}}$ has a right
adjoint given by
$\oneAr{\ple{\Id{},r\circ\blank}}{\Qtl^{(\blank)}}{R_\Qtl^{(\blank)}}$.
Hence, by \refToProp{adj-to-modality}, there is \acomonadic
operator \TNat{\Lbang}{\Qtl^{(\blank)}}{\Qtl^{(\blank)}} given by
$\Lbang_X \alpha = \iota\circ r\circ \alpha$, 
for any set $X$ and $\alpha \in \Qtl^X$. 

Recall that the doctrine $\Qtl^{(\blank)}$ carries a much richer structure induced
from that of the original quantale $\Qtl$: 
for any set $X$, \ple{\Qtl^X,\Lsup, \order_X,\Ltensor_X,\Lunit_X} is a commutative quantale with
the pointwise structure and, 
for $\alpha,\beta\in \Qtl^X$, the operation $\alpha\Llimpl_X\beta\colon= \Lsup
\{\zeta\in\Qtl^X\mid \alpha\Ltensor_X\zeta\order_X \beta\}$ determines an adjoint pair
$\blank\Ltensor_X\alpha\dashv\alpha\Llimpl_X\blank$; \ie for every $\gamma\in\Qtl^X$, one
has that
$\alpha\Ltensor_X\gamma\order_X\beta$ if and only if $\gamma\order_X \alpha\Llimpl_X\beta$. 
Furthermore, the \comonadic operator \TNat{\Lbang}{\Qtl^{(\blank)}}{\Qtl^{(\blank)}}
enjoys additional properties: 
for any set $X$ and $\alpha,\beta\in\Qtl^X$, we have 
$\Lbang_X \alpha \order_X \Lunit_X$ and
$\Lbang_X \alpha\order_X \Lbang_X\alpha \Ltensor_X\Lbang_X\alpha$, and  
$\Lunit_X \order_X \Lbang_X \Lunit_X$ and 
$\Lbang_X\alpha \Ltensor_X \Lbang_X \beta \order_X \Lbang_X (\alpha \Ltensor_X \beta)$. 
Therefore, the indexed poset $\Qtl^{(\blank)}$ provides a model of first order
intuitionistic linear logic, where $\Lbang$ is the linear exponential modality.
\end{exa}

\begin{exas}
Let \fun{\Doc}{\C\op}{\Pos} be a doctrine.
The propositional connectives are defined in terms of adjunctions
involving \Doc and another doctrine defined from it where the
adjoint functors between the base categories are the identity, see
\cite{LawvereF:adjif}, see also
\cite{JacobsB:catltt,MaiettiME:exacf}.
So \refToProp{adj-to-modality} provides \comonadic
operators associated with each connectives.
Two interesting instances are the following:
\begin{enumerate}
\item Consider the doctrine \fun{\Doc^2}{\C\op}{\Pos}, defined
by $\Doc^2X = \Doc X \times \Doc X$ and $\DReIdx{\Doc^2}{f} =
\DReIdx{\Doc}{f} \times \DReIdx{\Doc}{f}$. Note that there is a
1-arrow
\oneAr{\ple{\Id{\C},\Delta}}{\Doc}{\Doc^2} where
$\Delta_X=\ple{\id{\Doc X},\id{\Doc X}}$.
Conjunction on \Doc is determined by
a right adjoint to \ple{\Id{\C},\Delta} in \IdxPos,
that is the octuple
\ple{\Id{\C},\Delta,\Id{\C},\Land,\id{\Id{C}},\id{\Id{C}}} is 
an adjunction between \Doc and $\Doc^2$.  
Hence, by \refToProp{adj-to-modality}, there is \acomonadic
operator on $\Doc^2$ given by
$\ple{\alpha,\beta}\mapsto\ple{\alpha\Land\beta,\alpha\Land\beta}$,
for $\alpha,\beta\in\Doc X$.  
\item Assume further that \C has finite products and consider an
object $X$ in \C.  
Consider the doctrine $\fun{\Doc^X}{\C\op}{\Pos}$, determined as
$\Doc^X(Y)=\Doc(Y\times X)$ and
$\DReIdx{\Doc^X}{(f)}=\DReIdx{\Doc}{(f\times\id{X})}$.
There is a 1-arrow $\oneAr{\ple{\Id{\C},p^X}}{\Doc}{\Doc^X}$ where
$p^X_Y=\DReIdx{\Doc}{\pi_1}$ and \fun{\pi_1}{Y\times X}{Y} is the
first projection.  
A universal quantifier  $\Forall X.$ on $\Doc$ over $X$ is a
right adjoint to $\ple{\Id{\C},p^X}$ in \IdxPos, \ie the octuple
\ple{\Doc,\Doc^X,\Id{\C}, p^X,\Id{\C},
\Forall X.,\id{\Id{\C}},\id{\Id{\C}}}
is an adjunction in \IdxPos.
Hence, by \refToProp{adj-to-modality}, there is \acomonadic
operator on $\Doc^X$ given as $\alpha\mapsto p^X(\Forall X.\alpha)$
for $\alpha\in\Doc^X(Y)=\Doc(Y\times X)$.   
\end{enumerate}
\end{exas}

We did not consider the other cases of connectives because the 
modality each of those induces is the identity as the next proposition
explains in a more general context.

\begin{prop}\label{prop:trivial-modality}
Let $\fun{\Doc}{\C\op}{\Pos}$ and $\fun{\aDoc}{\C\op}{\Pos}$ be
doctrines on the same base category. Suppose
\ple{\Doc,\aDoc,\Id{\C},\LAdjLift,\Id{\C},\RAdjLift,
\id{\Id{\C}},\id{\Id{\C}}} is an adjunction. 
Then, for each object $X$ in \C, the following hold:
\begin{enumerate}[{\normalfont(i)}]
\item\label{prop:trivial-modality:1} $\LAdjLift_X \cdot \RAdjLift_X
  \cdot \LAdjLift_X = \LAdjLift_X$ and $\RAdjLift_X \cdot \LAdjLift_X
  \cdot \RAdjLift_X = \RAdjLift_X$;
\item\label{prop:trivial-modality:2} $\LAdjLift_X \cdot \RAdjLift_X =
  \id{\aDoc X}$ if and only if $\RAdjLift_X$ is injective if and only
  if $\LAdjLift_X$ is surjective; 
\item\label{prop:trivial-modality:3} $\RAdjLift_X \cdot \LAdjLift_X =
  \id{\Doc X}$  if and only if $\LAdjLift_X$ is injective if and only
  if $\RAdjLift_X$ is surjective. 
\end{enumerate}
\end{prop}
\begin{proof}
(i) is immediate since the adjunction $\LAdjLift_X\dashv\RAdjLift_X$
involves posetal categories. (ii) and (iii) follow directly from (i).
\end{proof}

The next step is an application of a remarkable result by \cite{HermidaC:fibacf} about
fibred adjunctions as it allows to show that any adjunction in \IdxPos can be \dfn{factored}
as the composition of two adjunctions where one is the identity adjunction on the base
categories. 
For this, recall that $\IdxPos$ has a vertical/cartesian factorisation
system, that is, any 1-arrow $\oneAr{\ple{F,f}}{\Doc}{\aDoc}$ from the 
doctrine $\fun{\Doc}{\C\op}{\Pos}$ to the doctrine $\fun{\aDoc}{\D\op}{\Pos}$ can be
factored by ``change of base''as $\ple{F,\id{\aDoc F\op}}  \circ \ple{\Id{\C},f}$
\[\xymatrix@C=6em@R=5ex{\C\op\ar@/^/[rd]^(.4)\Doc_(.7){}="f"
\ar[d]_{\Id{\C}\op}\\
\C\op\ar[r]^(.35){\aDoc F\op}^(.6){}="g"_(.6){}="k"
\ar[d]_{F\op}&\Pos\\
\D\op\ar@/_/[ru]_(.4){\aDoc}^(.7){}="h"
&\ar"f";"g"_{f}|{\smash{\tdot}}
\ar"k";"h"_{\id{}}|{\smash{\tdot}}}\]
The factorization of the adjunction follows this decomposition for the left adjoint.
Recall Lemma~3.2 from \cite{HermidaC:fibacf} in the case of doctrines.

\begin{lem}\label{lem:adj-factorization}
Let \ple{\C,\D,\LAdj,\RAdj,\unit,\counit} be an adjunction in \Ct{Cat}. If
\fun{\aDoc}{\D\op}{\Pos} is a doctrine, then there is an adjunction 
\ple{\aDoc\LAdj\op,\aDoc,\LAdj,\id{},\RAdj,\aDoc\counit\op,\unit,\counit} in \IdxPos as
depicted in the diagram
\begin{equation}\label{equfx}\xymatrix@C=4.5em{
\aDoc\LAdj\op\ar@/^8pt/[rr]^{\ple{\LAdj,\id{}}}\ar@{}[rr]&& 
\aDoc.\ar@/^8pt/[ll]^{\ple{\RAdj,\aDoc\counit\op}}}
\end{equation}
\end{lem}

\begin{proof}
We apply \refToProp{idxpos-adj} to show 
$\ple{\aDoc\LAdj\op,\aDoc,\LAdj,\id{\aDoc\LAdj\op},\RAdj,
\aDoc\counit\op,\unit,\counit}$ is an adjunction in \IdxPos. 
Since \ple{\C,\D,\LAdj,\RAdj,\unit,\counit} is already an adjunction in
\Ct{Cat}, it remains to check the natural transformations 
\TNat{\unit}{\Id{\C}}{\RAdj\LAdj}
and
\TNat{\counit}{\LAdj\RAdj}{\Id{\D}}
determine 2-arrows in \IdxPos as follows
\[\begin{array}{l@{\kern2em}l}
\twoAr{\unit}{\ple{\Id{\C},\id{\aDoc\LAdj\op}}}
{\ple{\RAdj\LAdj,\aDoc(\counit\LAdj)\op}}&
\twoAr{\counit}{\ple{\LAdj\RAdj,\aDoc\counit\op}}
{\ple{\Id{\D},\id{\aDoc}}.}
\end{array}\]
In other words, the inequalities
\[\id{\aDoc\LAdj X}\order
\DReIdx{\aDoc\LAdj}{\unit_X}\circ\DReIdx{\aDoc}{\counit_{\LAdj X}}\qquad
\DReIdx{\aDoc}{\counit_Y}\order\DReIdx{\aDoc}{\counit_Y}\]
hold for each object $X$ in \C and $Y$ in $\D$.
They are in fact identities: the second is immediate, and the first follows 
from the triangular identity (\ref{triai}) for an adjunction
\begin{equation}\label{tadj}
\DReIdx{\aDoc\LAdj}{\unit_X}\circ\DReIdx{\aDoc}{\counit_{\LAdj X}}=
\DReIdx{\aDoc}{(\counit_{\LAdj X}\circ\LAdj\unit_X)}=
\DReIdx{\aDoc}{\id{\LAdj X}} = \id{\aDoc\LAdj X} 
\end{equation}
by functoriality of $Q$.
\end{proof}

Theorem~3.4 in \cite{HermidaC:fibacf} restricted to the case of doctrines is the
following.

\begin{thm}\label{thm:adj-factorization}
Let \fun{\Doc}{\C\op}{\Pos} and \fun{\aDoc}{\D\op}{\Pos} be
doctrines, and suppose the octuple
\ple{\Doc,\aDoc,\LAdj,\LAdjLift,\RAdj,\RAdjLift,\unit,\counit} is an
adjunction in \IdxPos. Then that adjunction factors through the adjunction
in $(\ref{equfx})$ as
\begin{equation}\label{equf}\xymatrix@C=4.5em{
\Doc\ar@/^8pt/[rr]^{\ple{\Id{\C},\LAdjLift}}\ar@{}[rr]&&
\aDoc\LAdj\op\ar@/^8pt/[ll]^{\ple{\Id{\C},(\Doc\unit\op)(\RAdjLift\LAdj\op)}}
\ar@/^8pt/[rr]^{\ple{\LAdj,\id{}}}\ar@{}[rr]&& 
\aDoc.\ar@/^8pt/[ll]^{\ple{\RAdj,\aDoc\counit\op}}}
\end{equation}
where the first one is \ple{\Doc,\aDoc\LAdj\op,\Id{\C},\LAdjLift,\Id{\C},
(\Doc\unit\op)(\RAdjLift\LAdj\op),\id{},\id{}}.
\end{thm}
\begin{proof}
We see the \ple{\Doc,\aDoc\LAdj\op,\Id{\C},\LAdjLift,\Id{\C},
(\Doc\unit\op)(\RAdjLift\LAdj\op),\id{\Id{\C}},\id{\Id{\C}}}
is an adjunction in \IdxPos as another application of \refToProp{idxpos-adj}.
Obviously \ple{\C,\C,\Id{\C},\Id{\C},\id{},\id{}} is the identity adjunction
in \Ct{Cat}. To check the natural transformation 
\TNat{\id{\Id{\C}}}{\Id{\C}}{\Id{\C}}
determines 2-arrows in \IdxPos
\[\begin{array}{l@{\quad\mbox{and}\quad}l}
\twoAr{\id{\Id{\C}}}{\ple{\Id{\C},\id{\Doc}}}
{\ple{\Id{\C},(\Doc\unit\op)(\RAdjLift\LAdj\op)\LAdjLift}}&
\twoAr{\id{\Id{\C}}}
{\ple{\Id{\C},\LAdjLift(\Doc\unit\op)(\RAdjLift\LAdj\op)}}
{\ple{\Id{\C},\id{\aDoc\LAdj\op}}}
\end{array}\]
we must see that the inequalities
\[\id{\Doc X}\order
\DReIdx{\Doc}{\unit_X}\circ\RAdjLift_{\LAdj X}\circ\LAdjLift_X\quad\mbox{and}\quad
\LAdjLift_X\circ\DReIdx{\Doc}{\unit_X}\circ\RAdjLift_{\LAdj X}
\order\id{\aDoc\LAdj X}\]
hold for each object $X$ in \C.
The first inequality holds since
$\twoAr{\unit}{\ple{\Id{\C},\id{\Doc}}}
{\ple{\RAdj\LAdj,(\RAdjLift\LAdj\op)\LAdjLift}}$
is a 2-arrow in \IdxPos.
For the second inequality, note that
$\LAdjLift_X\circ\DReIdx{\Doc}{\unit_X}\circ\RAdjLift_{\LAdj X}=
\DReIdx{\aDoc}{\LAdj\unit_X}\circ\LAdjLift_{\RAdj\LAdj X}\circ
\RAdjLift_{\LAdj X}$ since \TNat{\LAdjLift}{\Doc}{\aDoc\LAdj\op}.
Since \twoAr{\counit}{\ple{\LAdj\RAdj,(\LAdjLift\RAdj\op)\RAdjLift}}
{\ple{\Id{\D},\id{\aDoc}}} is a 2-arrow in \IdxPos, we have that
$\LAdjLift_{\RAdj\LAdj X}\circ\RAdjLift_{\LAdj X}\order
\DReIdx{\aDoc}{\counit_{\LAdj X}}$. Now the result follows from (\ref{tadj}).

To see that the composition of the two adjunctions gives the
original adjunction, note that the
top and bottom compositions in (\ref{equf})
give the top and bottom 1-arrow in
\[\xymatrix@C=4.5em{
\Doc\ar@/^6pt/[r]^{\ple{\LAdj,\LAdjLift}}&
\aDoc.\ar@/^6pt/[l]^{\ple{\RAdj,\RAdjLift}}}\]
It is immediate to see that
$\ple{\LAdj,\id{}}\cdot\ple{\Id{\C},\LAdjLift}=
\ple{\LAdj,\LAdjLift}$. For the other composition, the first
components coincide trivially, and for the second components
apply the commutativity of the following diagram of natural
transformations
\[\xymatrix@=5.5em{\aDoc\ar@/^10pt/[rr]^-{\RAdjLift}
\ar[d]^-{\aDoc\counit\op}\ar[r]_-{\RAdjLift}&
\Doc\RAdj\op\ar[d]_-{\Doc\RAdj\op\counit\op}
\ar[r]_-{P((\RAdj\counit)(\unit\RAdj))\op}&
\Doc\RAdj\op\\  
\aDoc\LAdj\op\RAdj\op\ar[r]^-{\RAdjLift\LAdj\op\RAdj\op}&
\Doc\RAdj\op\LAdj\op\RAdj\op\ar[ru]_-{\Doc\unit\op\RAdj\op}}\]
where the square commutes by naturality of \RAdjLift, the right-hand
triangle by functoriality of \Doc, and the top triangle by one of the
triangular identities for adjunctions (\ref{triai}). Finally one sees
immediately the compositions of the 2-arrows give the 2-arrows of the
original adjunction.
\end{proof}

\begin{cor}\label{cor:adj-to-modality}
Let \fun{\Doc}{\C\op}{\Pos} and \fun{\aDoc}{\D\op}{\Pos} be
doctrines, and suppose the octuple
\ple{\Doc,\aDoc,\LAdj,\LAdjLift,\RAdj,\RAdjLift,\unit,\counit} is an
adjunction in \IdxPos. Then
$\necmod=\LAdjLift\cdot(\Doc\unit\op)\cdot(\RAdjLift\LAdj\op)$ is \acomonadic operator on
the doctrine \fun{\aDoc\LAdj\op}{\C\op}{\Pos}.
\end{cor}
\begin{proof}
It follows immediately applying \refToProp{adj-to-modality} to the
first adjunction in (\ref{equf}).
\end{proof}

\begin{exa}\label{ex:presh}
Let \C and \D be category with pullbacks, and let
\ple{\C,\D,\LAdj,\RAdj,\unit,\counit} be an adjunction where
\fun{\LAdj}{\C}{\D} preserves pullbacks. As in \refToEx{pb-cat}, there
is an adjunction
\ple{\Sub[\C]{},\Sub[\D]{},\LAdj,\LAdjLift,
\RAdj,\RAdjLift,\unit,\counit} on the doctrines of subobjects.
By \refToCor{adj-to-modality}, there is \acomonadic operator on
the doctrine \fun{\Sub[\D]{}\LAdj\op}{\C\op}{\Pos},  
defined as $\necmod_X\alpha=\LAdj\alpha'$,
where $X \in \C_0$ and $\alpha \in \Sub[\D]{\LAdj X}$ and
$\alpha'\in\Sub[\C]{X}$ is defined by the following pullback diagram
\[\xymatrix@C=3em{P\ \ar[d]\ar@{^{(}->}[r]^{\alpha'}
\ar@{}[dr]|{\text{p.b.}}&
X \ar[d]^{\unit_X}\\
\RAdj A\ \ar@{^{(}->}[r]_{\RAdj \alpha}&\RAdj\LAdj X}\]
The construction is reminiscent of that of a
modal operator from a geometric morphism between elementary toposes,
see the original paper \cite{GhilardiS:modtpc}, or Section~10.1 in
\cite{BraunerG07}, and also \cite{Reyes91,AwodeyBS02,AwodeyB03}.  
Indeed, a geometric morphism from the topos \topos to the topos
\atopos is an adjunction
\ple{\topos,\atopos,\LAdj,\RAdj,\unit,\counit} such that the left
adjoint $\LAdj$ preserves finite limits.

The paradigmatic example of a \comonadic operator obtained from a geometric morphism is that
offered by presheaves over a category \C. Recall that the category of presheaves over \C
is the functor category \Psh{\C}. If we let $\C_0$ be the discrete category of the objects
of \C and write \fun{i}{\C_0}{\C} the inclusion functor, post-composition with it determines
a functor \fun{\LAdj=\blank\circ i\op}{\Psh{\C}}{\Psh{\C_0}} which preserves all limits
and colimits as these are computed pointwise---although $\C_0=\C_0\op$ we maintain the
redundant notation $\C_0\op$ just for mental hygiene. Since the functor category
\Psh{\C} is complete and has a generating set, \LAdj has a right adjoint
\fun{\RAdj}{\Psh{\C_0}}{\Psh{\C}}. Hence, $\LAdj\dashv\RAdj$ is a geometric morphism, thus
it induces an \comonadic operator on
\[\xymatrix{\Psh{\C}\op\ar[rr]^{\Sub[\Psh{\C_0}]{}\LAdj\op}\ar[rd]_{\LAdj\op}&&\Pos\\
&\Psh{\C_0}\ar[ru]_{\Sub[\Psh{\C_0}]{}}}\]

Finally, note that, if $\KFr = \ple{W,R}$  is a Kripke frame with $R$ reflexive and
transitive, taking $\C = \KFr\op$, the above geometric morphism provides another way to
construct Kripke models categorically.
In detail, a presheaf $D$ over $\KFr\op$ specifies, for each world $w\in W$, a set $D(w)$,
modelling individuals which exist at the world $w$, 
and, for each $wRv$, a function \fun{D_{wv}}{D_w}{D_v}, describing how individuals
existing at the world $w$ ``evolve'' in the world $v$. 
A ``formula'' $\alpha$ on $D$ is a family of subsets, that is, for each world $w\in W$, $\alpha_w\subseteq D_w$, and 
the modal operator identifies those formulas which are subpresheaves of $D$, namely, those $\alpha$ such that, for all $w,v\in W$, if $wRv$ then $\alpha_w \subseteq D_{wv}^{-1}(\alpha_v)$.
\end{exa}

We conclude this section showing that the construction in \refToCor{adj-to-modality}
extends to a 2-functor \fun{\AMFun}{\Adj{\IdxPos}}{\NecIdxPos}.

For an adjunction \adj in \IdxPos write
\[\necmod^\adj\colon=
\LAdjLift^\adj\cdot(\Doc^\adj(\unit^\adj)\op)\cdot(\RAdjLift^\adj(\LAdj^\adj)\op)\]
which is \acomonadic operator by \refToCor{adj-to-modality}. Let
$\AMFun(\adj)=\ple{\aDoc^\adj(\LAdj^\adj)\op,\necmod^\adj}$.
For a 1-arrow \oneAr{\ple{F,f,G,g,\theta}}{\adj}{\aadj}, let
\begin{equation}\label{ontf}
\AMFun(\ple{F,f,G,g,\theta})\colon=\ple{F,g(\LAdj^\adj)\op}.
\end{equation}
For a 2-arrow
\twoAr{\ple{\alpha,\beta}}{\ple{F,f,G,g,\theta}}{\ple{F',f',G',g',\theta'}},
let
\begin{equation}\label{twtf}
\AMFun(\ple{\alpha,\beta})\colon=\alpha.
\end{equation}

\begin{prop}\label{prop:2-fun-adj-modality}
With the assignments above, \fun{\AMFun}{\Adj{\IdxPos}}{\NecIdxPos} is a 2-functor.
\end{prop}

\begin{proof}
We just have to check that the identities in (\ref{ontf}) and (\ref{twtf}) determine
arrows in \NecIdxPos, as the algebraic identities will then follow immediately.
Since
$\TNat{g(\LAdj^\adj)\op}{\aDoc^\adj(\LAdj^\adj)\op}{\aDoc^\aadj (G\LAdj^\adj)\op}$
and $G\LAdj^\adj=\LAdj^\aadj F$ by \refToProp{catadj}, in order to see that
\[\oneAr{\ple{F,g(\LAdj^\adj)\op}}
{\ple{\aDoc^\adj(\LAdj^\adj)\op,\necmod^\adj}}
{\ple{\aDoc^\aadj(\LAdj^\aadj)\op,\necmod^\aadj}}\]
ia a 1-arrow in \NecIdxPos we are left to check that
for every object $X$ in the base category of $\aDoc(\LAdj^\adj)\op$, we have 
\[g_{\LAdj^\adj X}\cdot\necmod^\adj_X\order\necmod^\aadj_{FX}\cdot g_{\LAdj^\adj X}.\]
In the diagram of natural transformations
\[\xymatrix@C=6em@R=.65em{
\aDoc^\adj(\LAdj^\adj)\op\arnta[r]^-{\RAdjLift^\adj(\LAdj^\adj)\op}
\arntb[dddddd]^-{g(\LAdj^\adj)\op}&
\Doc^\adj(\RAdj^\adj\LAdj^\adj)\op\arnta[r]^-{\Doc^\adj(\unit^\adj)\op}
\arnta[ddd]_{f(\RAdj^\adj\LAdj^\adj)\op}\ar@{{}{}}[rddd]|{*}&
\Doc^\adj\arnta[r]^-{\LAdjLift^\adj\ }\arnta[ddd]_-{f}&
\aDoc^\adj(\LAdj^\adj)\op\arnta[dd]^-{g(\LAdj^\adj)\op}\\&&& \\
&&&\aDoc^\aadj(G\LAdj^\adj)\op\\ 
&
\Doc^\aadj(F\RAdj^\adj\LAdj^\adj)\op\arnta[r]^-{\Doc^\aadj(F\unit^\adj)\op}&
\Doc^\aadj F\op\arnta[r]^-{\kern1.3ex\LAdjLift^\aadj F\op}&
\Doc^\aadj(\LAdj^\aadj F)\op\ar@{2-2-}[u]\\ \\ \\
\aDoc^\aadj(\LAdj^\adj)\op\arnta[r]^-{\Doc^\aadj\theta(\LAdj^\adj)\op}
\arnta[rddd]_-{\RAdjLift^\aadj(\LAdj^\aadj F)\op}
\ar@{{}{}}[ruuuuuu]|(.3)*=0[@]{\geq}&
\Doc^\aadj(\RAdj^\aadj G\LAdj^\adj)\op\arntb[uuu]^-{\Doc^\aadj(\theta\LAdj^\adj)\op}
\arnta[uuur]^-{\Doc^\aadj(\theta\LAdj^\adj\cdot F\unit^\adj)\op}\\ \\ \\
&\Doc^\aadj(\RAdj^\aadj\LAdj^\aadj F)\op\arnta[uuuuuur]_-{\Doc^\aadj(\unit^\aadj F)\op}
}\]
the marked square commutes by naturality of $f$, the triangle by functoriality of
$\Doc^\aadj$, and all the other paths commutes (possibly up to inequality as shown) by
\refToProp{catadj}.

Given now a 2-arrow
$\twoAr{\ple{\alpha,\beta}}{\ple{F,f,G,g,\theta}}{\ple{F',f',G',g',\theta'}}$ in 
\Adj{\IdxPos} to see that
$\twoAr{\alpha}{\ple{F,g\LAdj^\adj}}{\ple{F',g'\LAdj^\adj}}$ is a
2-arrow in \NecIdxPos, we have to show that, for every object $X$ in the base category
of $\aDoc(\LAdj^\adj)\op$, it is the case that 
$g_{\LAdj^\adj X}\order\DReIdx{\aDoc'}{\LAdj^\aadj\alpha_X}\cdot g'_{\LAdj^\adj X}$. 
By \refToProp{catadj}, the equality $\LAdj^\aadj\alpha=\beta\LAdj^\adj$ holds  and,  
since \twoAr{\beta}{\ple{G,g}}{\ple{G',g'}} in \IdxPos, we obtain that
$g_{\LAdj^\adj X}\order\DReIdx{\aDoc'}{\beta_{\LAdj^\adj X}}\cdot g'_{\LAdj^\adj X}$, as needed.
\end{proof}

\begin{exa}
A particular example of \comonadic operators is found in the categorical semantics of the
linear exponential modality (a.k.a. bang modality) of propositional linear logic provided by
linear-nonlinear adjunctions.  
A \dfn{linear-nonlinear adjunction} is a monoidal adjunction between a symmetric monoidal
category and a cartesian category; the induced comonad on the symmetric monoidal
category interprets the bang modality, see \cite{Benton94}.
The categorical notion swiftly extends to doctrines where the construction in
\refToCor{adj-to-modality} provides a model of the bang modality in a higher order setting.
The role of the cartesian category is played by a
\dfn{primary doctrine}, see \eg \cite{EmmeneggerJ:eledac}), that is, a doctrine
\fun{\Doc}{\C\op}{\Pos} where \C has finite products and, for each object $X$ in \C, the
fiber $\Doc X$ carries an inf-semilattice structure preserved by reindexing.
The role of the symmetric monoidal category is played by a
\emph{$($symmetric$)$ monoidal doctrine}, which one defines following the work on
monoidal indexed categories of \cite{MoellerV20}. We give some of the details in
Appendix~\ref{append}, but shall develop fully the particular instance of \comonadic
operators in a subsequent paper.
\end{exa}

\section{\Comonadic modalities from comonads}\label{sect:comonad}

As is well-known, there is a deep connection between comonads and adjunctions in a
2-category: every adjunction determines a comonad.  Viceversa, when the 2-category admits
the Eilenberg-Moore construction for comonads, a comonad generates an adjunction. This
connection is particularly interesting when we consider a left exact comonad \cmdfun on a
topos \topos: the category of coalgebras \EM{\topos}{\cmdfun} is a topos and the
Eilenberg-Moore adjunction between \EM{\topos}{\cmdfun} and \topos is a geometric
morphism, see \eg\cite{MacLaneM92}.
As we have seen in \refToEx{presh}, geometric morphisms generate \comonadic
operators; hence, combining these two facts, we obtain that a left exact comonads on an
elementary topos determines \acomonadic operator.

In this section, we study the relationship between adjunctions and comonads in the
2-category \IdxPos of doctrines, showing how comonads generate adjunctions, as
expected, and \comonadic operators from those. We start by determining comonads in
\IdxPos.

\begin{prop}\label{prop:comonad}
Let \fun{\Doc}{\C\op}{\Pos} be a doctrine. 
A comonad on \Doc is completely determined by a quadruple
$\cmd=\ple{\cmdfun,\cmdlift,\comul,\cun}$ where 
\begin{enumerate}[{\normalfont(i)}]
\item \ple{\cmdfun,\comul,\cun} is a comonad on \C;
\item\label{prop:comonad:2}
\oneAr{\ple{\cmdfun,\cmdlift}}{\Doc}{\Doc} is a 1-arrow in \IdxPos;
\item\label{prop:comonad:3}
 \twoAr{\comul}{\ple{\cmdfun,\cmdlift}}{\ple{\cmdfun^2,(\cmdlift\cmdfun\op)\cmdlift}}
and \twoAr{\cun}{\ple{\cmdfun,\cmdlift}}{\ple{\Id{\C},\id{}}} are 2-arrows in \IdxPos.
\end{enumerate}
\end{prop}

\begin{proof}
Straightforward.
\end{proof}

\begin{rem}
More explicitly, \refToPropItem{comonad}{2} requires that
\TNat{\comul}{\Doc}{\Doc\cmdfun\op} 
and \refToPropItem{comonad}{3} states that, for each object $X$ in  
\C, the following inequalities hold
\[\cmdlift_X\order\DReIdx{\Doc}{\comul_X}\circ \cmdlift_{\cmdfun X}\circ \cmdlift_X
\qquad\mbox{ and }\qquad
\cmdlift_X\order\DReIdx{\Doc}{\cun_X}.\]
\end{rem}

For abstract reasons, a comonad in \IdxPos always admits the Eilenberg-Moore construction,
see \cite{BlackwellR:twodmt}. Here we limit ourselves to present the 
direct computation of the Eilenberg-Moore object for a comonad
$\cmd=\ple{\cmdfun,\cmdlift,\comul,\cun}$ on the doctrine
\fun{\Doc}{\C\op}{\Pos}.
The Eilenbeerg-Moore object for \cmd can be given on the doctrine
\fun{\EM{\Doc}{\cmd}}{\big(\EM{\C}{\cmdfun}\big)\op}{\Pos} defined as follows.

The category $\EM{\C}{\cmdfun}$ is the category of coalgebras for the comonad
$\ple{\cmdfun,\comul,\cun}$ on \C, namely, 
objects are pairs $\ple{C,c}$ where $C$ is an object in \C and $\oneAr{c}{C}{\cmdfun C}$
is an arrow in \C such that the diagram
\[\xymatrix{&C\ar[d]^-{c}\ar[ld]_{\id{C}}\ar[r]^-{c}&\cmdfun C\ar[d]^{\comul_C}\\
C&\cmdfun C\ar[l]^{\cun_C}\ar[r]_{\cmdfun c}&\cmdfun\cmdfun C}\]
commutes, and an arrow \fun{f}{\ple{C,c}}{\ple{C',c'}} is an arrow \oneAr{f}{C}{C'} in \C,
such that
\[\xymatrix{C\ar[d]_-{c}\ar[r]^-{f}&C'\ar[d]^{c'}\\
\cmdfun C\ar[r]^{\cmdfun f}&\cmdfun C'.}\]

With the intention to produce the doctrine
\fun{\EM{\Doc}{\cmd}}{\big(\EM{\C}{\cmdfun}\big)\op}{\Pos}, 
for each coalgebra \ple{C,c} let \EM{\Doc}{\cmd}\ple{C,c} be the suborder of
$\Doc C$ on the subset 
\set{\alpha\in\Doc C}{\alpha\order\DReIdx{\Doc}{c}(\cmdlift_C(\alpha))}.

Given an arrow \oneAr{f}{\ple{C,c}}{\ple{C',c'}} in \EM{\C}{\cmdfun} and
$\beta\in\EM{\Doc}{\cmd}\ple{C',c'}$, note that
$\beta\order\DReIdx{\Doc}{c'}(\cmdlift_{C'}(\beta))$ by definition of
\EM{\Doc}{\cmd}\ple{C',c'}. Thus
\[\begin{array}{r@{}l}\DReIdx{\Doc}{f}(\beta)\order
\DReIdx{\Doc}{f}(\DReIdx{\Doc}{c'}(\cmdlift_{C'}(\beta))){}&{}
=\DReIdx{\Doc}{(c'f)}(\cmdlift_{C'}(\beta)))
=\DReIdx{\Doc}{(fKc)}(\cmdlift_{C'}(\beta)))\\[1ex]
&{}=\DReIdx{\Doc}{c}(\DReIdx{\Doc}{\cmdfun (f)}(\cmdlift_{C'}(\beta)))
=\DReIdx{\Doc}{c}(\cmdlift_C(\DReIdx{\Doc}{f}(\beta))).
\end{array}\]
So \DReIdx{\Doc}{f} sends elements of \EM{\Doc}{\cmd}\ple{C',c'} to elements of \EM{\Doc}{\cmd}\ple{C,c}: let
\DReIdx{\EM{\Doc}{\cmd}}{f} be the restriction of \DReIdx{\Doc}{f}. It follows immediately
that \EM{\Doc}{\cmd} is a doctrine.

\begin{rem}
Note that the inequality $\DReIdx{\Doc}{c}(\cmdlift_C(\alpha)) \order \alpha$ holds
for every $\alpha \in \Doc C$, by properties of $c$ and $\cun_C$.
Hence the elements of $\EM{\Doc}{\cmd}\ple{C,c}$ are the fixpoints of
$\DReIdx{\Doc}{c}\circ\cmdlift_C$.
Furthermore, as we shall see, $\DReIdx{\Doc}{c}\circ\cmdlift_C$ is an idempotent on
$\Doc C$ (it is a consequence of \refToProp{cmd-to-modality}).
Thus, as in \Pos idempotents split, one gets $\EM{\Doc}{\cmd}\ple{C,c}$ by splitting
$\DReIdx{\Doc}{c}\circ\cmdlift_C$.
\end{rem}

Next we introduce the \dfn{forgetful} 1-arrow
\oneAr{\ple{\Forget,\ForgetLift}}{\EM{\Doc}{\cmd}}{\Doc} as follows:
the functor $\fun{\Forget}{\EM{\C}{\cmdfun}}{\C}$ is the actual forgetful functor
from the category of coalgebras; the natural transformation
\TNat{\ForgetLift}{\EM{\Doc}{\cmd}}{\Doc(\Forget)\op} is given by the inclusion of
$\EM{\Doc}{\cmd}\ple{C,c}$ into $\Doc C$ as \ple{C,c} varies among the objects of
\EM{\C}{\cmdfun}.
It is immediate to see the functor $\Forget$ is faithful and, for each object \ple{C,c} in
$\EM{\C}{\cmdfun}$, the map $\ForgetLift_{\ple{C,c}}$ is injective.

Finally the universal 2-arrow
\twoAr{\unita}{\ple{\Forget,\ForgetLift}}
{\ple{\cmdfun,\cmdlift}\ple{\Forget,\ForgetLift}} as requested in 
(\ref{udata}) is given by the family $\unita$ given by
\[\unita_{\ple{C,c}}\colon=\fun{c}{C}{\cmdfun C},\quad
\mbox{ as }\ple{C,c}\mbox{ varies among the objects in }\EM{\C}{\cmdfun}.\]
One sees immediately that \TNat{\unita}{\Forget}{\cmdfun\Forget}. It determines an
appropriate 2-arrow 
in \IdxPos because for any $\alpha\in\EM{\Doc}{\cmd}\ple{C,c}$, by definition of
\EM{\Doc}{\cmd}\ple{C,c} one has that
\[\alpha\order\DReIdx{\Doc}{c}(\cmdlift_C(\alpha))
=\left(\DReIdx{\Doc}{\unita_{\ple{C,c}}} \circ \cmdlift\big(\Forget\big)\op_{\ple{C,c}}\right)(\alpha)\]
After introducing the dramatis person\ae, we are ready to prove the characterization
of the Eilenberg-Moore construction for a comonad in \IdxPos.

\begin{thm}\label{thm:em-obj}
Let $\fun{\Doc}{\C\op}{\Pos}$ be a doctrine and $\cmd$ a comonad on $\Doc$.
Then
\[\xymatrix@C=5em@R=1em{&\Doc\ar[dd]^-{\ple{\cmdfun,\cmdlift}}_(.6){}="a"\\
\EM{\Doc}{\cmd}\ar[ru]^{\ple{\Forget,\ForgetLift}}
\ar[rd]_-{\ple{\Forget,\ForgetLift}}^{}="b"\\  
\ar@{=>}"b";"a"^-{\unita}&\Doc}\]
is the Eilenberg-Moore construction for $\cmd$ in \IdxPos.
\end{thm}

\begin{proof}
We begin the proof analysing the data for the 2-problem in \refToDefItem{adjcom}{3}: one
has an arbitrary doctrine \fun{\aDoc}{\D\op}{\Pos} and a diagram of 1-arrows and
2-arrows in \IdxPos
\begin{equation}\label{fcc}
\vcenter{\xymatrix@C=5em@R=1em{&\Doc\ar[dd]^-{\ple{\cmdfun,\cmdlift}}_(.6){}="a"\\
\aDoc\ar[ru]^{\ple{X,x}}
\ar[rd]_-{\ple{X,x}}^{}="b"\\  
\ar@{=>}"b";"a"^-{\xi}&\Doc}}
\end{equation}
where the pair \ple{\ple{X,x},\xi} satisfies the two commutativity conditions in
(\ref{1cohe}). These translate precisely in the commutative diagrams of natural
transformations
\begin{equation}\label{transl1}
\vcenter{\xymatrix{X\arnta[r]^-{\xi}\arnta[d]_-{\xi}&
\cmdfun X\arntb[d]^-{\comul X}\\
\cmdfun X\arnta[r]_-{\cmdfun\xi}&\cmdfun\cmdfun X}}\qquad
\vcenter{\xymatrix{X\arnta[r]^-{\xi}\arnta[rd]_-{\id{X}}&
\cmdfun X\arntb[d]^-{\cun X}\\&X}}
\end{equation}
while the condition on the 2-arrow in (\ref{fcc}) requires that the natural transformation
\TNat{\xi}{X}{\cmdfun X} is such that, for every object $D$ in \D and $\beta\in \aDoc(D)$, we have 
\begin{equation}\label{transl2}
x_D(\beta)\order \DReIdx{\Doc}{(\xi_D)}(\cmdlift_{X(D)}(x_D(\beta))).
\end{equation}
In turn, the commutativity of the two diagrams (\ref{transl1}) is equivalent to requiring
that, for every object $D$ in \D, there is a structure of coalgebra \ple{X(D),\xi_D} for the
comonad \ple{\cmdfun,\comul,\cun} on the object $X(D)$ in the category \C, and that, for every
arrow \fun{f}{D}{D'} in \D, the arrow \fun{X(f)}{\ple{X(D),\xi_D}}{\ple{X(D'),\xi_{D'}}} is
a homomorphism of coalgebras. At the same 
time, condition (\ref{transl2}) is equivalent to requiring that the monotone function
\fun{x_D}{\aDoc(X(D))}{\Doc(X(D))} factors through
\[\xymatrix@R=.8ex@C=2em{\aDoc(X(D))\ar[rrd]_-{x_D}\ar[rrr]^-{x_D}&&&\Doc(X(D))\\
&&\EM{\Doc}{\cmd}\ple{X(D),\xi_D}\ar@{^(->}[ru].}\]
Hence the data for the 2-problem determine precisely a
1-arrow \oneAr{\ple{\overline{\ple{X,\xi}},x}}{Q}{\EM{\Doc}{\cmd}} ensuring uniqueness,
and it is immediate to check that the required diagram commutes.

Similarly, for an arrow \oneAr{\gamma}{\ple{\ple{X,x},\xi}}{\ple{\ple{Y,y},\upsilon}} of the 2-problem, that is, a 2-arrow \twoAr{\gamma}{\ple{X,x}}{\ple{Y,y}} in \IdxPos, the commutative
diagram (\ref{2cohe}) determines precisely a natural transformation
\TNat{\overline{\gamma}}{\overline{\ple{X,\xi}}}{\overline{\ple{Y,\upsilon}}}; the
inequality encoded in the 2-arrow \twoAr{\gamma}{\ple{X,x}}{\ple{Y,y}} in \IdxPos is the
same as that encoded in the 2-arrow
\twoAr{\overline{\gamma}}{\ple{\overline{\ple{X,\xi}},x}}
{\ple{\overline{\ple{Y,\upsilon}},y}} in \IdxPos.
\end{proof}

\begin{cor}\label{cor:cmd-to-adj}
Let \fun{\Doc}{\C\op}{\Pos} be a doctrine and $\cmd=\ple{\cmdfun,\cmdlift,\comul,\cun}$ be
a comonad on $\Doc$. Then there is an adjunction
$\adj^\cmd=
\ple{\EM{\Doc}{\cmd},\Doc,\Forget,\ForgetLift,\hat{\cmdfun},\cmdlift,\unit^\cmd,\cun}$
between \EM{\Doc}{\cmd} and \Doc.   
\end{cor}
\begin{proof}
It follows from \refToThm{em-obj} and general results in \cite{Street72}. But we make
explicit each component of the adjunction as is obtained from the general case.
Among the data determining the adjunction, only two may need to be described:
the functor \fun{\hat{\cmdfun}}{\C}{\EM{\C}{\cmdfun}} is the free coalgebra functor and
gives, for an object $X$ in \C, the free coalgebra 
$\hat{\cmdfun}X=\ple{\cmdfun X,\comul_X}$. The natural transformation is the canonical
embedding of a coalgebra into the free coalgebra
\TNat{\unit^\cmd}{\Id{\EM{\C}{\cmdfun}}}{\hat{\cmdfun}\Forget} defined as 
$\unit^\cmd_{\ple{X,c}}=c$.
\end{proof}

In fact, in the general 2-adjunction between comonads and adjunctions in a 2-category \K 
when \K admits the Eilenberg-Moore construction, as in diagram (\ref{maind}), we know that
the Eilenberg-Moore construction gives the right 2-adjoint from the 2-category \Cmd{\K} of
comonads in \K.
So we briefly collect the  
data for the 2-category \Cmd{\IdxPos} in order to apply that result in the present
situation. The 2-category \Cmd{\IdxPos} has
\begin{description}
\Item [objects] which are pairs \ple{\Doc,\cmd} where \Doc is a doctrine and \cmd is a
comonad on \Doc;
\Item [1-arrows] from \ple{\Doc,\cmd} to \ple{\aDoc,\acmd}, with
$\cmd=\ple{\cmdfun,\cmdlift,\comul^\cmd,\cun^\cmd}$ and
$\acmd=\ple{\acmdfun,\acmdlift,\comul^\acmd,\cun^\acmd}$, consist of a 1-arrow 
\oneAr{\ple{F,f}}{\Doc}{\aDoc} and a 2-arrow
\twoAr{\theta}{\ple{F\cmdfun,(f\cmdfun\op)\cmdlift}}{\ple{\acmdfun F,(\acmdlift F\op)f}}
in \IdxPos such that the following diagrams of functors and natural transformations commute:
\[\xymatrix{
F\cmdfun \ar[r]^{\theta} \ar[dr]_{F\cun^\cmd} & \acmdfun F \ar[d]^{\cun^\acmd F} \\
& F 
}\qquad
\xymatrix{
F \cmdfun \ar[d]_{F\comul^\cmd} \ar[rr]^{\theta} && \acmdfun F \ar[d]^{\comul^\acmd F} \\
F \cmdfun^2 \ar[r]_{\theta \cmdfun} & \acmdfun F \cmdfun \ar[r]_{\acmdfun \theta} & \acmdfun^2 F
}\]
\Item [2-arrows] from \ple{\ple{F,f},\theta} to \ple{\ple{G,g},\zeta}, which are 1-arrows
from \ple{\Doc,\cmd} to \ple{\aDoc,\acmd}, with
$\cmd=\ple{\cmdfun,\cmdlift,\comul^\cmd,\cun^\cmd}$ and
$\acmd=\ple{\acmdfun,\acmdlift,\comul^\acmd,\cun^\acmd}$, consist of a 2-arrow
\twoAr{\alpha}{\ple{F,f}}{\ple{G,g}} such that the following diagram of functors and
natural transformations commutes
\[ \xymatrix{
F\cmdfun \ar[r]^{\alpha \cmdfun} \ar[d]_{\theta} & G\cmdfun \ar[d]^{\zeta} \\
\acmdfun F \ar[r]_{\acmdfun \alpha} & \acmdfun G
} \]
\end{description}

The instance of diagram (\ref{maind}) which we have been addressing is the
following
\[ \xymatrix{
\IdxPos \ar@/^8pt/[rr]^{\IncFun} \ar@{}[rr]|{\bot}  &&
\Cmd{\IdxPos} \ar@/^8pt/[ll]^{\EMFun} \ar@/_8pt/[rr]_{\EMAFun} \ar@{}[rr]|{\bot} &&
\Adj{\IdxPos} \ar@/_8pt/[ll]_{\CmdFun} 
} \]
Since by \refToCor{adj-to-modality} every adjunction between doctrines induces \acomonadic
operator, via \EMAFun one obtains \acomonadic operator also from a comonad.

\begin{prop}\label{prop:cmd-to-modality} 
Let $\fun{\Doc}{\C\op}{\Pos}$ be a doctrine and $\cmd=\ple{\cmdfun,\cmdlift,\comul,\cun}$
a comonad on $\Doc$. Then, the natural transformation
$\TNat{\necmod^\cmd}{\Doc\Forget}{\Doc\Forget}$,
defined, for 
each coalgebra \ple{X,c} in $\EM{\C}{\cmdfun}$, by
$\necmod^\cmd_{\ple{X,c}}=\DReIdx{\Doc}{c}\circ\cmdlift_X$,
is \acomonadic operator on \fun{\Doc{\Forget}\op}{\EM{\C}{\cmdfun}}{\Pos}.
\end{prop}
\begin{proof}
By \refToCor{cmd-to-adj}, \ple{\Forget,\ForgetLift,\hat{\cmdfun},\cmdlift,\unit^\cmd,\cun}
is an adjunction between \EM{\Doc}{\cmd} and \Doc. 
By \refToCor{adj-to-modality}, $\necmod^\cmd= \ForgetLift \cdot (\EM{\Doc}{\cmd}\unit^\cmd)
\cdot (\cmdlift\Forget)$ is \acomonadic operator on
$\fun{\Doc{\Forget}\op}{\EM{\C}{\cmdfun}\op}{\Pos}$,
but, for each coalgebra \ple{X,c} in $\EM{\C}{\cmdfun}$, $\unit^\cmd_{\ple{X,c}} = c$ and
$\Forget\ple{X,c} = X$, $\DReIdx{\EM{\Doc}{\cmd}}{c} = \DReIdx{\Doc}{c}$ by definition, and
$\ForgetLift$ is an inclusion.
\end{proof}

\begin{exa}\label{ex:preshcont}
An interesting case of \refToProp{cmd-to-modality} is that of toposes of presheaves as
models of first order modal logic.
We have already seen in \refToEx{presh} how one obtains \acomonadic operator
\[\xymatrix{\Psh{\C}\op\ar[rr]^{\Sub[\Psh{\C_0}]{}\LAdj\op}\ar[rd]_{\LAdj\op}&&\Pos\\
&\Psh{\C_0}\ar[ru]_{\Sub[\Psh{\C_0}]{}}}\]
on the category of presheaves \Psh{\C} from the adjunction which is the geometric morphism
\begin{equation}\label{geomor}
\xymatrix@=4em{\Psh{\C_0}\ar@<.5ex>@/^/[r]^{\RAdj}\ar@{{}{}}[r]|-{\top}
&\Psh{\C}\ar@<.5ex>@/^/[l]^{\blank\circ i\op}}
\end{equation}
where $\C_0$ denotes the discrete category of the objects of \C and \fun{i}{\C_0}{\C} is
the inclusion functor. But the category of presheaves is exactly the category of
coalgebras for the comonad determined by the adjunction (\ref{geomor}), see
\cite{JohnstoneP:skeeli}; so \refToProp{cmd-to-modality} applies, and the modal
operator obtained on a presheaf model is obtained directly from the subobject doctrine on
\Psh{\C_0}]{} and the geometric morphism that determines the presheaves as coalgebras.
\end{exa}

\section{The global picture}\label{sect:picture}

\refToProp{2-fun-adj-modality} produces a construction of \acomonadic operator from
adjunctions as a 2-functor $\fun{\AMFun}{\Adj{\IdxPos}}{\NecIdxPos}$. And
\refToProp{cmd-to-modality} describes the action of the composition \CMFun in the diagram
\[\xymatrix@=5em{\Cmd{\IdxPos}\ar@<-.5ex>@/_/[r]_-{\EMAFun}\ar@{{}{}}[r]|-{\bot}
\ar[rd]_-{\CMFun}
&\Adj{\IdxPos}\ar[d]^-{\AMFun}\ar@<-.5ex>@/_/[l]_-{\CmdFun}\\
&\NecIdxPos.}\]
The goal of this section is to complete the above diagram, by showing that $\AMFun$ is
part of a local adjunction, see \cite{BettiP88}. Hence so is $\CMFun$.

We start by comparing the 2-functor \AMFun to the composite $\CMFun \circ \CmdFun$, both constructing a doctrine with interior operator from an adjunction in \IdxPos. 
They do not coincide, but can be canonically compared by a 2-natural transformation. 
Recall that \CMFun maps a comonad \ple{\Doc,\cmd}, for
$\cmd=\ple{\cmdfun,\cmdlift,\comul^\cmd\,\cun^\cmd}$,
to the doctrine with \acomonadic operator \ple{\Doc(\Forget)\op,\necmod^\cmd} where
$\necmod^\cmd_{\ple{X,c}}=\DReIdx{\Doc}{c}\cdot\cmdlift$.

Since \AMFun is a 2-functor, its action on the unit of the 2-adjunction $\CmdFun\dashv\EMAFun$
produces a natural comparison $\AMFun(\adj)\to\CMFun(\CmdFun(\adj))$ for
$\adj=\ple{\Doc,\aDoc,\LAdj,\LAdjLift,\RAdj,\RAdjLift,\unit,\counit}$
an adjunction in \IdxPos. 

Indeed, let $\cmd\colon=\CmdFun(\adj) = 
\ple{\LAdj\RAdj,(\LAdjLift\RAdj\op)\RAdjLift,\LAdj\unit\RAdj,\counit}$ be the induced
comonad on $\aDoc$. The component of the unit of the 2-adjunction on $\adj$ is given
by the 1-arrow \oneAr{\ple{K,k,\Id{},\id{},\id{}}}{\adj}{\EMAFun(\cmd)}, where 
\oneAr{\ple{K,k}}{\Doc}{\EM{\aDoc}{\cmd}} is the comparison 1-arrow given by the
Eilenberg-Moore construction. The 1-arrow  \ple{K,k} is obtained by the universal property
of $\EM{\aDoc}{\cmd}$ applied to the following diagram: 
\[\vcenter{\xymatrix@C=4em@R=1ex{
&\aDoc\ar[dd]^-{\ple{\LAdj\RAdj,(\LAdjLift\RAdj\op)\RAdjLift}}_(.55){}="t"\\
\Doc\ar[ru]^-{\ple{\LAdj,\LAdjLift}}\ar[rd]_-{\ple{\LAdj,\LAdjLift}}^(.68){}="s"\\
\ar@{=>}"s";"t"^-{\LAdj\unit}&\aDoc}}\]
More explicitly,  \ple{K,k} is defined as follows: 
$KX \colon= \ple{\LAdj X,\LAdj\unit_X}$, for each object $X$ in the base category of $\Doc$, 
$Kf \colon= \LAdj f$, for each arrow in the base category of $\Doc$, and $k = \LAdjLift$. 
This is well-defined thanks to the following chain of inequalities: 
\[\LAdjLift_X\order\LAdjLift_X\circ\DReIdx{\Doc}{\unit_X}\circ\RAdjLift_{\LAdj X}\circ\LAdjLift_X 
=\DReIdx{\aDoc}{(\LAdj\unit_X)}\circ
((\LAdjLift\RAdj\op)\cdot\RAdjLift)_{\LAdj X}\circ\LAdjLift_X.\] 

\begin{prop}\label{prop:modality-comparison}
Let $\adj=\ple{\Doc,\aDoc,\LAdj,\LAdjLift,\RAdj,\RAdjLift,\unit,\counit}$ be an adjunction
in \IdxPos, and consider
$\cmd\colon=\ple{\LAdj\RAdj,(\LAdjLift\RAdj\op)\RAdjLift,\LAdj\unit\RAdj,\counit}$ 
the associated comonad on the doctrine $\aDoc$.
Let $\ple{K,k}$ be the comparison 1-arrow. 
Then,
\oneAr{\ple{K,\id{}}}{\ple{\aDoc\LAdj\op,\necmod^\adj}}{\ple{\aDoc(\Forget)\op,\necmod^\cmd}}
is a 1-arrow in \IdxPos and $\necmod^\adj=\necmod^\cmd K$. 
\end{prop}

\begin{proof}
It is immediate since, for each object $X$, 
$\necmod^\adj_X = \LAdjLift_X \DReIdx{\Doc}{\unit_X} \RAdjLift_{\LAdj X} =
\DReIdx{\aDoc\LAdj}{\unit_X} \LAdjLift_{\RAdj\LAdj X} \RAdjLift_{\LAdj X} =  
\necmod^\cmd_{K X}$.
\end{proof}

Finally, let us note that this comparison 1-arrow is a component of a 2-natural
transformation, obtained by postcomposition of the unit of the 2-adjunction
$\CmdFun\dashv\EMAFun$ with the 2-functor $\AMFun$.

In order to show that \AMFun is part of a local adjunction, 
We start by constructing a comonad from an object \ple{\Doc,\necmod} in \NecIdxPos.

\begin{prop}\label{prop:modality-to-cmd}
Let $\fun{\Doc}{\C\op}{\Pos}$ be a doctrine and $\TNat{\necmod}{\Doc}{\Doc}$ be
\acomonadic operator on $\Doc$.
Then, \ple{\Id{\C},\necmod,\id{},\id{}} is a comonad on $\Doc$. 
\end{prop}
\begin{proof}
There is only to check that \twoAr{\id{}}{\ple{\Id{\C},\necmod}}{\ple{\Id{\C},\id{}}}
and \twoAr{\id{}}{\ple{\Id{\C},\necmod}}{\ple{\Id{\C},\necmod \cdot \necmod}} are
well-defined 2-arrows. But, for each object $X$ in \C, $\necmod_X\order\id{\Doc X}$ and
$\necmod_X\order \necmod_X \cdot \necmod_X$ hold  by \refToDef{modality}.
\end{proof}
In other words, \refToProp{modality-to-cmd} shows that \acomonadic operator on a
doctrine \Doc is exactly a vertical comonad on it.

We introduce the 2-functor $\fun{\MCFun}{\NecIdxPos}{\Cmd{\IdxPos}}$ by letting, for
\ple{\Doc,\necmod} a doctrine with \comonadic operator,
$\MCFun(\ple{\Doc,\necmod})\colon= \ple{\Doc,\Id{},\necmod,\id{},\id{}}$, which is a
comonad by \refToProp{modality-to-cmd};
for a 1-arrow \oneAr{\ple{F,f}}{\ple{\Doc,\necmod^\Doc}}{\ple{\aDoc,\necmod^\aDoc}}
$\MCFun(\ple{F,f})\colon= \ple{F,f,\id{}}$;
for a 2-arrow \twoAr{\theta}{\ple{F,f}}{\ple{G,g}}
$\MCFun(\theta)\colon= \theta$. 

\begin{prop}\label{prop:2-fun-modality-to-cmd}
With the assignments above, \fun{\MCFun}{\NecIdxPos}{\Cmd{\IdxPos}}  is a 2-functor. 
\end{prop}
\begin{proof}
The proof is straightforward. 
The only interesting part is checking that it is well-defined on the 1-arrows. Indeed,
for each object $X$ in the base category \C of the doctrine $\Doc$, we have  
$f_X\cdot\necmod^\Doc_X\order\necmod^\aDoc_{FX}\cdot f_X$, by definition of 1-arrow
in \NecIdxPos. And this ensures that 
\twoAr{\id{}}{\ple{F,f\cdot \necmod^\Doc}}{\ple{F,(\necmod^\aDoc F\op)\cdot f}} is a
2-arrow in \IdxPos.
\end{proof}  
It is easy to see that the 2-functor $\MCFun$ is full and faithful. 
Hence the 2-category $\NecIdxPos$ is isomorphic to the 2-category of
vertical comonads in $\IdxPos$.

Now let \fun{\MAFun}{\NecIdxPos}{\Adj{\IdxPos}} be the composition 
$\xymatrix {\NecIdxPos \ar[r]^{\MCFun} & \Cmd{\IdxPos} \ar[r]^{\EMAFun} & \Adj{\IdxPos}}$
which sends an object \ple{\Doc,\necmod} in \NecIdxPos to the Eilenberg-Moore adjunction 
of the associated comonad $\MCFun(\Doc,\necmod)=\ple{\Id{\C},\necmod,\id{},\id{}}$
\[\xymatrix@=5em{
\necmod\Doc\ar@<.5ex>@/^/[r]^{\ple{\Id{\C},\ForgetLift}}\ar@{}[r]|{\bot}
&\Doc \ar@<.5ex>@/^/[l]^{\ple{\Id{\C},\necmod}}}\]
where, from the general construction in (\ref{maind}), the Eilenberg-Moore object
\fun{\necmod\Doc}{\C\op}{\Pos} for the comonad induced by $\necmod$ is 
$\necmod\Doc X=\set{\alpha \in \Doc X}{\alpha=\necmod_X\alpha}$. Also
$\DReIdx{\necmod\Doc}{f}=\DReIdx{\Doc}{f}$, and $\TNat{\ForgetLift}{\necmod\Doc}{\Doc}$ is
the inclusion.  

\begin{thm}\label{thm:2-adj-am-ma}
There is a local adjunction $\MAFun \dashv \AMFun$, where 
\begin{itemize}
\item the unit $\TNat{\Delta}{\Id{\NecIdxPos}}{\AMFun \cdot \MAFun}$ is the identity lax
2-natural transformation, and  
\item the counit $\TNat{\nabla}{\MAFun \cdot \AMFun}{\Id{\Adj{\IdxPos}}}$ is given, for
an adjunction $\adj=\ple{\Doc,\aDoc,\LAdj,\LAdjLift,\RAdj,\RAdjLift,\unit,\counit}$
where $\fun{\Doc}{\C\op}{\Pos}$ and $\fun{\aDoc}{\D\op}{\Pos}$, by 
$\nabla_\adj=\ple{\Id{\C},(\Doc\unit\op)\cdot(\RAdjLift\LAdj\op),\LAdj,\id{},\unit}$, as
in the following diagram
\[\xymatrix{
\necmod\aDoc\LAdj\op \ar[dd]_{\ple{\Id{\C},(\Doc\unit\op)\cdot(\RAdjLift\LAdj\op)}} 
\ar@/^8pt/[rr]^{\ple{\Id{\C},\ForgetLift}} \ar@{}[rr]|{\bot}  && 
\aDoc\LAdj\op \ar@/^8pt/[ll]^{\ple{\Id{\C},\necmod^\adj}} \ar[dd]^{\ple{\LAdj,\id{}}}  \\
 && \\
\Doc  \ar@/^8pt/[rr]^{\ple{\LAdj,\LAdjLift}} \ar@{}[rr]|{\bot} 
\xtwocell[rruu]{}<>{_\unit} 
&&  \aDoc \ar@/^8pt/[ll]^{\ple{\RAdj,\RAdjLift}} 
} \]
and, for each 1-arrow $\oneAr{\phi}{\adj}{\aadj}$,  
$\nabla_\phi = \ple{\id{},\id{}}$.  
\end{itemize}
\end{thm}
\begin{proof}
The fact that $\Delta$ is a well-defined lax 2-natural transformation is straightforward, since $\AMFun \cdot \MAFun = \Id{\NecIdxPos}$. 
We check that $\nabla_\adj$ is a 1-arrow from $\MAFun(\ple{\aDoc\LAdj\op,\necmod^\adj})$ to $\adj$. 
We have $\ple{\LAdj \circ \Id{\C},\LAdjLift \cdot (\Doc\unit\op)\cdot(\RAdjLift\LAdj\op)} = \ple{\LAdj \circ \Id{\C},\id{} \cdot \ForgetLift}$, since, for each object $X$ in \C and $\alpha \in \necmod\aDoc\LAdj\op X$, we have $\LAdjLift_X(\DReIdx{\Doc}{\unit_X}(\RAdjLift_{\LAdj X}(\alpha))) = \necmod^\adj_X \alpha = \alpha$, by definition of $\necmod\aDoc\LAdj\op$. 
Then, we have to check that $\twoAr{\unit}{\ple{\Id{\C} \circ \Id{\C},(\Doc\unit\op)\cdot (\RAdjLift\LAdj\op) \cdot \necmod^\adj}}{\ple{\RAdj\LAdj,(\RAdjLift\LAdj\op)\cdot \id{}}}$ is a 2-arrow in \IdxPos, 
but this holds because $\TNat{\unit}{\Id{\C}}{\RAdj\LAdj}$ is a natural transformation and , for each object $X$ in \C, $\necmod^\adj_X \order \id{\aDoc\LAdj\op X}$, hence we get 
$\DReIdx{\Doc}{\unit_X} \circ \RAdjLift_{\LAdj X} \circ \necmod^\adj_X \order \DReIdx{\Doc}{\unit_X} \circ \RAdjLift_{\LAdj X}$. 

Now, consider a 1-arrow $\oneAr{\phi = \ple{F, f, G, g, \theta}}{\adj}{\aadj}$ in $\Adj{\IdxPos}$;
hence, we have $\MAFun(\AMFun(\phi)) = \ple{F, g(\LAdj^\adj)\op, F, g(\LAdj^\adj)\op, \id{}}$, and
we have to show that 
\[
\twoAr{\nabla_\phi = \ple{\id{},\id{}}}{\ple{F, f, G, g, \theta} \circ \nabla_\adj}{\nabla_\aadj \circ \ple{F, g(\LAdj^\adj)\op, F, g(\LAdj^\adj)\op, \id{}} }
\]
is a 2-arrow in $\Adj{\IdxPos}$. 
To this end, it is enough to prove that
\[\twoAr{\id{}}
{\ple{F,f\cdot(\Doc^\adj(\unit^\adj)\op)\cdot(\RAdjLift^\adj(\LAdj^\adj)\op)}}
{\ple{F,(\Doc^\aadj(\unit^\aadj)\op F\op)\cdot
(\RAdjLift^\aadj(\LAdj^\aadj)\op F\op)\cdot(g(\LAdj^\adj)\op)}}\]
and 
\[\twoAr{\id{}}{\ple{G\LAdj^\adj,g(\LAdj^\adj)\op}}{\ple{\LAdj^\aadj F,g(\LAdj^\adj)\op}}\]
are 2-arrows in \IdxPos, since the other conditions are trivially satisfied as the two
components are identities. 
The second is a 2-arrow since, by definition of 1-arrow in $\Adj{\IdxPos}$, the equality  
$G\LAdj^\adj = \LAdj^\aadj F$ holds.  
To see that so is the first, consider the following inequalities for $X$ an object in \C:
\begin{align*}
f_X \circ \DReIdx{\Doc^\adj}{\unit^\adj_X} \circ \RAdjLift^\adj_{\LAdj^\adj X} 
&= \DReIdx{\Doc^\aadj}{F\unit^\adj_X} \circ f_{\RAdj^\adj\LAdj^\adj X} \circ \RAdjLift^\adj_{\LAdj^\adj X} 
&& f\text{ is natural} \\
&\order \DReIdx{\Doc^\aadj}{F\unit^\adj_X} \circ \DReIdx{\Doc^\aadj}{\theta_{\LAdj^\adj X}} \circ \RAdjLift^\aadj_{G\LAdj^\adj X} \cdot g_{\LAdj^\adj X} 
&& \theta \text{ is a 2-arrow in \IdxPos} \\
&= \DReIdx{\Doc^\aadj}{\unit^\aadj_{FX}} \circ \RAdjLift^\aadj_{\LAdj^\aadj FX} \circ g_{\LAdj^\adj X} 
&& (\theta\LAdj^\adj)(F\unit^\adj) = \unit^\aadj F \text{ and } G\LAdj^\adj = \LAdj^\aadj F 
\end{align*}
Finally, we have the check the adjunction triangular laws:
$(\AMFun \nabla)(\Delta \AMFun) =
\Id{\AMFun}$ and $(\nabla \MAFun)(\MAFun \Delta) = \Id{\MAFun}$. 
The former holds as $\AMFun (\nabla_\adj)$
is the identity on $\AMFun(\adj)$ for any adjunction $\adj$.
The latter holds because, for any object \ple{\Doc,\necmod} in \NecIdxPos,
$\nabla_{\MAFun(\ple{\Doc,\necmod})}$ is the 
identity on $\MAFun(\ple{\Doc,\necmod})$, since
$\MAFun(\ple{\Doc,\necmod})$ is the Eilenberg-Moore adjunction of the comonad
\ple{\Id{},\necmod,\id{},\id{}} on $\Doc$.
\end{proof}

Now recall that, by definition, we have $\MAFun = \EMAFun\cdot\MCFun$ and observe that
$\CmdFun\cdot\EMAFun=\Id{\Cmd{\IdxPos}}$. Hence  $\MCFun = \CmdFun \circ \MAFun$.
Therefore, $\MCFun \dashv \CMFun$ is a local adjunction, as stated in the following
corollary.

\begin{cor}\label{cor:picture}
There is a diagram of $($lax$)$ 2-adjunctions
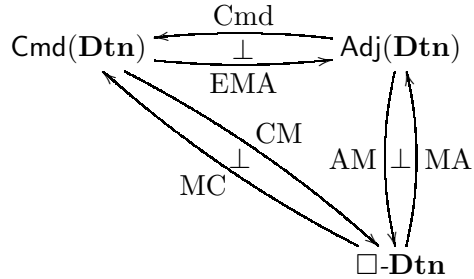
\[\xymatrix@=6em{\Cmd{\IdxPos}\ar@/_/[r]_-{\EMAFun}\ar@{{}{}}[r]|-{\bot}
\ar@/^/[rd]^-{\CMFun}
&\Adj{\IdxPos}\ar@/_/[d]_-{\AMFun}\ar@/_/[l]_-{\CmdFun}\ar@{}[d]|*=0[@]{\bot}\\
&\NecIdxPos\ar@/_/[u]_-{\MAFun}\ar@{}[lu]|*=0[@]{\bot}
\ar@/^/[lu]^-{\MCFun} }\]
where the diagonal adjunction is the composite of the other two.
\end{cor}

Finally we refine \refToThm{adj-factorization}, providing a new factorization through the
doctrine $\necmod\aDoc\LAdj\op$.

\begin{thm}\label{thm:adj-factorization-2}
Let $\fun{\Doc}{\C\op}{\Pos}$ and $\fun{\aDoc}{\D\op}{\Pos}$ be doctrines and consider an adjunction $\ple{\LAdj, \LAdjLift, \RAdj, \RAdjLift, \unit, \counit}$ between them. 
Then, the following diagram (of adjunctions)
\[ \xymatrix@R+15pt@C+5pt{
&&& \ar@/_8pt/[llld]_*=/^8pt/[flip][@]{\ple{\Id{\C},(\Doc\unit\op)(\RAdjLift\LAdj\op)}} \necmod\aDoc\LAdj\op 
\ar@/_8pt/[dd]_{\ple{\Id{\C},\ForgetLift}} \ar@{}[dd]|{\dashv} 
\ar@/_8pt/[rrrd]_*=/^8pt/[@]{\ple{\LAdj,\ForgetLift}} \ar@{}[rrrd]|*=0[@]{\top} &&& \\ 
\Doc \ar@/_8pt/[rrru]_*=/^8pt/[@]{\ple{\Id{\C},\LAdjLift}}  \ar@{}[rrru]|*=0[@]{\top}  
\ar@/^8pt/[rrrd]^*=/^8pt/[@]{\ple{\Id{\C},\LAdjLift}} \ar@{}[rrrd]|*=0[@]{\bot}   &&&
&&& 
\aDoc \ar@/_8pt/[lllu]_*=/^8pt/[flip][@]{\ple{\RAdj,\necmod(\aDoc\counit\op)}} 
\ar@/^8pt/[llld]^*=/dl8pt/[flip][@]{\ple{\RAdj,\aDoc\counit\op}}   \\
&&& \ar@/^8pt/[lllu]^*=/dr8pt/[flip][@]{\ple{\Id{\C},(\Doc\unit\op)(\RAdjLift\LAdj\op)}} 
\aDoc\LAdj\op \ar@/_8pt/[uu]_{\ple{\Id{\C},\necmod}} 
\ar@/^8pt/[rrru]^*=/^8pt/[@]{\ple{\LAdj,\id{}}} \ar@{}[rrru]|*=0[@]{\bot}    &&& 
} \]
commutes. Moreover $\TNat{\LAdjLift}{\Doc}{\necmod\aDoc\LAdj\op}$ is surjective and
$\TNat{(\Doc\unit\op)(\RAdjLift\LAdj\op)}{\necmod\aDoc\LAdj\op}{\Doc}$ is injective. 
\end{thm}
\begin{proof}
The commutativity of the diagram follows immediately from the definition of $\necmod$ and
\refToPropItem{trivial-modality}{1} and \refToThm{adj-factorization}. 
The fact that, for each object $X$, the function
$\fun{\LAdjLift_X}{\Doc X}{\necmod\aDoc\LAdj\op X}$ is
surjective and
$\fun{\DReIdx{\Doc}{\unit_X}\RAdjLift_{\LAdj X}}{\necmod\aDoc\LAdj\op X}{\Doc X}$ is
injective, follows from \refToPropItem{trivial-modality}{2}, noting that 
$\necmod_X = \LAdjLift_X \circ \DReIdx{\Doc}{\unit_X} \circ \RAdjLift_{\LAdj X}$ is the
identity on $\necmod\aDoc\LAdj\op X$ by definition.
\end{proof}

\newcommand{\inv}{^{\ast}} 
\newcommand{\proj}{\mathrm{pr}} 
\newcommand{\Str}{\ensuremath{\mathsf{Str}}\xspace}
\newcommand{\OFTree}{\ensuremath{\mathsf{Tr}}\xspace}
\newcommand{\N}{\mathbb{N}} 
\newcommand{\pfun}[3]{\ensuremath{{#1}:{#2}\rightharpoonup{#3}}}

\begin{exa}[Temporal Logics]
Consider the standard powerset doctrine \fun{\pw{}}{\Set\op}{\Pos}, sending a set $X$ to
the powerset $\pw{X}$ and a function \fun{t}{X}{Y} to the inverse image function
\fun{t\inv}{\pw{Y}}{\pw{X}}, and a 1-arrow \oneAr{\ple{F,f}}{\pw{}}{\pw{}}. 
Suppose that \fun{F}{\Set}{\Set} is an accessible functor, hence it admits a free comonad
(\cf \cite{GhaniLMP01}) \fun{\cmdfun^F}{\Set}{\Set}. We recall the construction in the
following.
\begin{itemize}
\item Given a set $A$, let $\cmdfun^F A = \nu X. A\times FX$ be the (underlying set of
the) final coalgebra for the functor
\[\xymatrix@C=4em@R=.1ex{\ct{Set}\ar[r]^-{A\times F}&\ct{Set}\\
X\ar@{|->}[r]&A\times FX}\]
and denote by \fun{\zeta_A}{\cmdfun^F A}{A\times F(\cmdfun^F A)} the structure map of the
final $A\times F$-coalgebra, which is an iso by the Lambek Lemma.
\item Since
\fun{\ple{\id{},\proj_2\circ\zeta_A}}{\cmdfun^FA}{\cmdfun^FA\times F(\cmdfun^FA)} is a
$\cmdfun^FA\times F$-coalgebra, there is a unique $\cmdfun^FA\times F$-coalgebra homomorphism
\fun{\cmu^F_A}{\cmdfun^FA}{\cmdfun^F\cmdfun^FA} such that the diagram
\[\xymatrix@C=4em@R=3em{\cmdfun^FA\ar[r]^-{\ple{\id{},\proj_2\circ\zeta_A}}\ar[d]_{\cmu^F_A}
&\cmdfun^FA\times F(\cmdfun^FA)\ar[d]^{\cmu^F_A\times\id{}}\\
\cmdfun^F\cmdfun^F A\ar[r]^-{\zeta_A}&\cmdfun^F\cmdfun^F\times F(\cmdfun^F\cmdfun^F A)}\]
commute.
\item Let \fun{\cun^F_A}{\cmdfun^FA}{A} be
$\cun^F_A=\proj_1\circ\zeta_A$.
\item Given a function \fun{t}{B}{A}, the function
\fun{\zeta_B}{\cmdfun^F B}{B\times F(\cmdfun^F B)} is a final $B\times F$-coalgebra; let
\fun{\cmdfun^F t}{\cmdfun^F B}{\cmdfun^F A} be the unique $A\times F$-homomorphism such
that the diagram
\[\xymatrix{\cmdfun^F B\ar[r]^{\zeta_B}\ar[d]_{\cmdfun^F t}
&B\times F(\cmdfun^F B)\ar[r]^{t\times \id{}}&A\times F(\cmdfun^F B)\ar[d]^{\cmdfun^F t}\\
\cmdfun^F A\ar[rr]^{\zeta_A}&&A\times F(\cmdfun^F A)\\}\]
commutes.
\end{itemize}
We can also define a natural transformation \TNat{\cmdlift_f}{\pw{}}{\pw{}{\cmdfun^F}\op} as
follows.
Consider a set $A$ and a subset $\alpha\in\pw{A}$.
We define a function \fun{\phi_\alpha}{\pw{\cmdfun^F A}}{\pw{\cmdfun^F A}} as 
$\phi_\alpha(\beta) = \zeta_A\inv(\alpha \times f_{\cmdfun^FA}(\beta))$, which is monotone
by construction, hence, since \pw{\cmdfun^FA} is a complete lattice, by the Knaster-Tarski
theorem, $\phi_\alpha$ has a greatest fixed point,
given by
$\nu\phi_\alpha=\bigcup\set{\beta\in\pw{\cmdfun^F A}}{\beta\subseteq\phi_\alpha(\beta)}$. 

Define $\cmdlift^f_A(\alpha)$ as $\nu\phi_\alpha$. 
This function is monotone, because, if $\alpha\subseteq\beta$, then $\nu\phi_\alpha =
\zeta_A\inv(\alpha\times f_{\cmdfun^FA}(\nu\phi_\alpha)) \subseteq \zeta_A\inv(\beta\times
f_{\cmdfun^FA}(\nu\phi_\alpha)) = \phi_\beta(\nu\phi_\alpha)$. Thus, by coinduction, we
get $\nu\phi_\alpha\subseteq\nu\phi_\beta$, as needed.  
In order to prove that $\cmdlift^f_A$ is natural in $A$, we have to check that, for each
function \fun{t}{B}{A} and $\alpha \in \pw{A}$,  
it is the case that $(\cmdfun^F t)\inv(\nu\phi_\alpha) = \nu\phi_{t\inf(\alpha)}$. 
First, note that 
\[\begin{split}
(\cmdfun^F t)\inv(\nu\phi_\alpha) &= (\cmdfun^F t)\inv(\zeta_A\inv(\alpha\times f_{\cmdfun^FA}(\nu\phi_\alpha))) \\ 
 &= (\zeta_A\circ\cmdfun^Ft)\inv (\alpha\times f_{\cmdfun^FA}(\nu\phi_\alpha)) \\ 
 &= ((\id{}\times F\cmdfun^Ft)\circ (t\times \id{}) \circ \zeta_B)\inv (\alpha\times f_{\cmdfun^FA}(\nu\phi_\alpha)) \\ 
 &= \zeta_B\inv (t\inv(\alpha)\times (F\cmdfun^Ft)\inv(f_{\cmdfun^FA}(\nu\phi_\alpha))) \\ 
 &= \zeta_B\inv (t\inv(\alpha)\times f_{\cmdfun^FB}((\cmdfun^Ft)\inv(\nu\phi_\alpha))) \\  
 &= \phi_{t\inv(\alpha)}((\cmdfun^F t)\inv(\nu\phi_\alpha)).
\end{split}\]
Hence, by coinduction, we get $(\cmdfun^Ft)\inv(\nu\phi_\alpha) \subseteq
\nu\phi_{t\inv(\alpha)}$. To prove the other inclusion,  we just have to prove that
$\cmdfun^Ft[\nu\phi_{t\inv(\alpha)}] \subseteq \nu\phi_\alpha$, where $\cmdfun^Ft[\beta]$
denotes the direct image of $\beta\in\pw{\cmdfun^FB}$ along $\cmdfun^Ft$.
To this end, we note that 
\[\begin{split}
\cmdfun^Ft[\nu\phi_{t\inv(\alpha)}] &
\subseteq \cmdfun^Ft[\zeta_B\inv(t\inv(\alpha)\times f_{\cmdfun^FB}(\nu\phi_{t\inv(\alpha)}))] \\
&=\cmdfun^Ft[((t\times\id{})\circ\zeta_B)\inv(\alpha\times f_{\cmdfun^FB}(\nu\phi_{t\inv(\alpha)}))] \\
&\subseteq \zeta_A\inv((\id{}\times F\cmdfun^Ft)[\alpha\times f_{\cmdfun^FB}(\nu\phi_{t\inv(\alpha)})]) \\ 
&= \zeta_A\inv(\alpha \times F\cmdfun^Ft[f_{\cmdfun^FB}(\nu\phi_{t\inv(\alpha)})]) \\
&\subseteq \zeta_A\inv(\alpha\times f_{\cmdfun^FA}(\cmdfun^Ft[\nu\phi_{t\inv(\alpha)}])) \\
&=\phi_\alpha(\cmdfun^Ft[\nu\phi_{t\inv(\alpha)}]).
\end{split}\]
To check that $\cmd^F = \ple{\cmdfun^F,\cmdlift^F,\cmu^F,\cun^F}$ is a comonad on \pw{}, it
is enough to show the following two inequalities: 
(1) $\cmdlift^F_A(\alpha) \subseteq (\cun^F_A)\inv(\alpha)$ and 
(2) $\cmdlift^F_A(\alpha) \subseteq
(\cmu^F_A)\inv(\cmdlift^F_{\cmdfun^FA}(\cmdlift^F_A(\alpha)))$ for all $\alpha\in\pw{A}$.

{\it Ad} (1) note that $\alpha\times f_{\cmdfun^FA}(\nu\phi_\alpha) \subseteq
\proj_1\inv(\alpha)$. Hence
$\nu\phi_\alpha = \zeta_A\inv(\alpha\times f_{\cmdfun^FA}(\nu\phi_\alpha)) \subseteq
\zeta_A\inv(\proj_1\inv(\alpha)) = (\cun^F_A)\inv(\alpha)$. 

{\it Ad} (2) we show $\cmu^F_A[\nu\phi_\alpha] \subseteq \nu\phi_{\nu\phi_\alpha}$. 
First of all, since $\alpha\times f_{\cmdfun^FA}(\nu\phi_\alpha) \subseteq
\proj_2\inv(f_{\cmdfun^FA}(\nu\phi_\alpha))$, we have
$\nu\phi_\alpha = \zeta_A\inv(\alpha\times f_{\cmdfun^FA}(\nu\phi_\alpha)) \subseteq
(\proj_2\circ\zeta_A)\inv(f_{\cmdfun^FA}(\nu\phi_\alpha))$.
Hence $\nu\phi_\alpha \subseteq 
\nu\phi_\alpha\cap (\proj_2\circ\zeta_A)\inv(f_{\cmdfun^FA}(\nu\phi_\alpha)) =
\ple{\id{},\proj_2\circ\zeta_A}\inv(\nu\phi_\alpha\times f_{\cmdfun^FA}(\nu\phi_\alpha))$. 
Therefore
\[\begin{split}
\cmu^F_A[\nu\phi_\alpha] &\subseteq  \cmu^F_A[ \ple{\id{},\proj_2\circ\zeta_A}\inv(\nu\phi_\alpha\times f_{\cmdfun^FA}(\nu\phi_\alpha)) ] \\
&\subseteq \zeta_{\cmdfun^FA}\inv( (\id{}\times F\cmu^F_A)[ \nu\phi_\alpha \times f_{\cmdfun^FA}(\nu\phi_\alpha) ]) \\
&= \zeta_{\cmdfun^FA}\inv( \nu\phi_\alpha \times F\cmu^F_A[ f_{\cmdfun^FA}(\nu\phi_\alpha) ] ) \\
&\subseteq \zeta_{\cmdfun^FA}\inv( \nu\phi_\alpha\times f_{\cmdfun^F\cmdfun^FA}( \cmu^F_A[\nu\phi_\alpha] ) ) \\
&= \phi_{\nu\phi_\alpha}(\cmu^F_A[\nu\phi_\alpha]).
\end{split}\]
Thus by coinduction we obtain (2).

Applying the construction in \refToProp{cmd-to-modality}, we obtain a
comonadic modal operator $\necmod^{\cmd^F}$ on the indexed poset
\fun{\aDoc}{(\EM{\Set}{\cmd^F})\op}{\Pos},
mapping a coalgebra \ple{A,c} for the comonad $\cmd^F$ to $\pw{A}$ and a coalgebra
morphism \fun{t}{\ple{B,d}}{\ple{A,c}} to the inverse image function
\fun{t\inv}{\pw{A}}{\pw{B}}. 
Explicitly, given a coalgebra \ple{A,c} and an element $\alpha\in\pw{A}$, we have 
$\necmod^{\cmd^F}_{\ple{A,c}} \alpha = c\inv(\cmdlift^F_A(\alpha)) = c\inv(\nu\phi_\alpha)$. 

This setting has a temporal interpretation: 
given the 1-arrow \ple{F,f}, the functor $F$ represents the ``branching type'', namely,
the branching structure of time, and $f$ lifts formulas to branches. 
The functor $\cmdfun^F$ models the whole time structure, that is, the present and all
possible futures, generated by the branching type $F$, and $\cmdlift^F$ lifts a formula to
time structures, basically, universally quantifying over time, according to $f$, roughly
saying that the formula holds in all possible future branches.
Given a coalgebra \ple{A,c} for the comonad $\cmd^F$, for each $x\in A$, $c(x)$ represents
the whole evolution of $x$  along time,  
hence, for each $\alpha\in\pw{A}$, we have $x\in \necmod^{\cmd^F}_{\ple{A,c}} \alpha$ if
all future evolutions of $x$ belongs to $\alpha$.
Therefore, roughly, $\necmod^{\cmd^F}$ is a generic kind of ``always'' modality, typical
of temporal logics. In the following we consider two explicit instances of this situation.
\end{exa}

\begin{exa}[Linear time] 
Consider $\ple{F,f}=\ple{\Id{},\id{}}$, that is, each instant has exactly one possible future. 
The free comonad is the stream comonad $\Str A = \nu X. A\times X = A^\omega$, mapping a set $A$ to the set $A^\omega$ of sequences of elements in $A$ indexed over natural numbers.
Given a sequence $a\in A^\omega$, we write $s_i$ to denote the $i$-th element of $s$, and $s[i..]$ to denote the sequence $r\in A^\omega$ such that $r_j = s_{j+i}$ for all $j\in\N$. 
Then, the counit maps $s$ to $s_0$ (the first element, namely the present) and the comultiplication maps $s$ to the sequence $(s[i..])_{i\in\N}$, namely the sequence of all suffixes of $s$. 

Let $\alpha\in\pw{A}$, we have $\cmdlift^F_A(\alpha) = \{ s\in A^\omega \mid s_i\in\alpha \mbox{ for all } i\in\N\}$, namely, the set of sequences where all elements belongs to/satisfies $\alpha$. 
Therefore, if \ple{A,c} is a coalgebra for \Str, $\necmod^\Str_{\ple{A,c}} \alpha = \{ x \in A \mid c(x)_i \in\alpha\mbox{ for all }i\in\N\}$, that is, it is the set of all elements $x\in A$ such that all its future instances (including the present one) belongs to $\alpha$. 

Therefore, $\necmod^\Str_{\ple{A,c}}$ provides a model for the ``globally'' ($\mathbf{G}$) modality of Linear Temporal Logic (LTL) \cite{BaierK08} and, moreover, 
the modality on the free coalgebra \ple{\Str A, \cmu^\Str_A} implements exactly the standard semantics of such a modality on infinite sequences.  
\end{exa}

\begin{exa} [Finitely ordered branching time]
Let  \fun{F}{\Set}{\Set} be the functor $FX = \bigcup_{n\in\N} X^n$. 
We can consider several natural transformations \TNat{f}{\pw{}}{\pw{}F\op} making
\ple{F,f} a 1-arrow. 
The two paradigmatic examples are the following: 
$f^\forall_A(\alpha) = \{ \ple{n,\ple{x_1,\ldots,x_n}} \in FX \mid x_i \in \alpha \mbox{ for all }  i \in 1..n\}$ and 
$f^\exists_A(\alpha) = \{ \ple{n,\ple{x_1,\ldots,x_n}} \in FX \mid x_i \in \alpha \mbox{ for some } i \in 1..n\}$.  

The free comonad is $\OFTree$, mapping a set $A$ to the set of finitely branching and
ordered trees labelled by $A$. 
Formally, such a tree is a partial function \pfun{t}{\N^\star}{A} with a
non-empty and prefix-closed domain such that, if $\ple{k_1,\ldots,k_n}\in \dom{t}$ and
$k\le k_n$, then $\ple{k_1,\ldots,k}\in \dom{t}$ (\cf \cite{Courcelle83,AczelAMV03}).  
The counit maps a tree $t$ to the label of its root, that is $t(\varepsilon)$, where
$\varepsilon$ is the empty sequence, and  
the comultiplication maps a tree $t$ to $\cmu^F_A(t)$ such that $\dom{\cmu^T_A(t)} =
\dom{t}$ and $\cmu^F_A(t)(u)$ is the subtree of $t$ rooted at $u\in\dom{t}$.  
The behaviour of the natural transformation $\cmdlift^F$ of course depends on $f$, for instance, 
for $f = f^\forall$, it maps $\alpha\in\pw{A}$ to the set of trees where all nodes have label in $\alpha$, while 
for $f = f^\exists$, it maps $\alpha\in\pw{A}$ to the set of trees containing an infinite path starting from the root where all nodes have label in $\alpha$. 

Then, given a coalgebra \ple{A,c} for the comonad \OFTree  and $\alpha\in\pw{A}$, we have
$x\in \necmod^\OFTree_{\ple{A,c}}\alpha$ if all nodes in $c(x)$ have label in $\alpha$,
when $f=f^\forall$, and if there is an infinite path in $c(x)$ where all nodes have label
in $\alpha$, when $f=f^\exists$. 
Therefore, $\necmod^\OFTree_{\ple{A,c}}$ provides a model for the modalities
``invariantly'' ($\mathbf{AG}$) and ``potentially always'' ($\mathbf{EG}$) of Computation
Tree Logic (CTL) \cite{BaierK08}, depending on the choice of $f$. 
\end{exa}

\section*{Acknowledgment}
The authors would like to thank Jacopo Emmenegger, Fabio Pasquali and Cosimo Perini Brogi
for many helpful discussions on the subject.

\appendix
\section{\Comonadic operators from linear-nonlinear adjunctions}\label{append}

A well-known approach to provide categorical semantics to the linear exponential modality
$\Lbang$---read as ``bang''---of propositional linear logic is by means of 
linear-nonlinear adjunctions as in \cite{Benton94}.  
A \dfn{linear-nonlinear adjuction} is a monoidal adjunction beween a symmetric monoidal
category and a cartesian category; the induced comonad on the symmetric monoidal
category interprets the bang modality.  
This notion is easily extended to doctrines where the construction in
\refToCor{adj-to-modality} provides a model of the bang modality in a higher order setting.  

In the present context, the role of the cartesian category is played by a
\dfn{primary doctrine}, that is, a doctrine
\fun{\Doc}{\C\op}{\Pos} 
where \C has finite products and, for each object $X$ in \C, the fiber $\Doc X$ carries an
inf-semilattice structure preserved by reindexing, see \eg \cite{EmmeneggerJ:eledac}.
The symmetric monoidal category turns into a
\emph{$($symmetric$)$ monoidal doctrine}, which we define below, following the definition
of monoidal indexed categories in \cite{MoellerV20}. We shall employ the 2-cartesian
structure of the 2-category \IdxPos. So, in the following, given indexed posets
\fun{\Doc}{\C\op}{\Pos} and \fun{\aDoc}{\D\op}{\Pos}, we denote by
\fun{\Doc\times\aDoc}{(\C\times\D)\op}{\Pos} the \emph{product} doctrine mapping a pair of
objects \ple{X,Y} to the product (in \Pos) 
$\Doc X\times\aDoc Y$ and acting similarly on arrows. Furthermore, we denote by $\One$ the
\emph{terminal} doctrine whose base is the terminal category and mapping its unique object
to the singleton poset. We shall write
\oneAr{\alpha_{P_1,P_2,P_3}}{P_1\times(P_2\times P_3)}{(P_1\times P_2)\times P_3},
\oneAr{\lambda_P}{\One\times P}{P}, \oneAr{\rho_P}{\One\times P}{P}, and
\oneAr{\sigma_{P_1,P_2}}{P_1\times P_2}{P_2\times P_1}  
for the usual 1-iso for associativity, left and right identity, and symmetry.

A \dfn{$($symmetric$)$ monoidal doctrine} consists of 
\begin{itemize}
\item a doctrine \fun{\aDoc}{\D\op}{\Pos}, 
\item two 1-arrows \oneAr{\ple{\otimes,\bullet}}{\aDoc\times\aDoc}{\aDoc} and 
\oneAr{\ple{I,\iota}}{\One}{\aDoc}, and 
\item four invertible 2-arrows 
\[\begin{array}{c@{\qquad}c}
\multicolumn2c{\twoAr{a}{\ple{\otimes,\bullet}\circ
(\ple{\otimes,\bullet}\times\ple{\Id{},\id{}})\circ\alpha_{Q,Q,Q}}
{\ple{\otimes,\bullet}\circ(\ple{\Id{},\id{}}\times\ple{\otimes,\bullet})}}\\[1ex]
\twoAr{l}{\ple{\otimes,\bullet}\circ(\ple{I,\iota}\times\ple{\Id{},\id{}})}
{\lambda_Q}&
\twoAr{r}{\ple{\otimes,\bullet}\circ(\ple{\Id{},\id{}}\times\ple{I,\iota})}
{\rho_Q}\\[1ex]
\multicolumn2c{\twoAr{s}{\ple{\otimes,\bullet}\circ\sigma_{Q,Q}}{\ple{\otimes,\bullet}}}
\end{array}\]
\end{itemize}
such that \ple{\D,\otimes,I,a,l,r,s} is a symmetric monoidal category. 
As the 2-arrows $a$, $l$, $r$ and $s$ are invertible, the inequalities they induce on the
fibres are actually equalities, namely, the following diagrams commute
\[\xymatrix@C=5em@R=1em{(\aDoc A \times \aDoc B) \times \aDoc C
\ar[dd]_{\bullet_{A,B} \times \id{}} \ar[r]^{(\alpha_{Q,Q,Q})_{A,B,C}}
&\aDoc A \times (\aDoc B \times \aDoc C) \ar[rd]^{\id{} \times \bullet_{B,C}}\\
&&\aDoc A \times \aDoc (B\otimes C) \ar[dd]^{\bullet_{A,B\otimes C}}\\
\aDoc (A\otimes B) \times \aDoc C \ar[dr]^{\bullet_{A\otimes B,C}} \\
&\aDoc ((A\otimes B)\otimes C) \ar[r]^{\DReIdx{\aDoc}{(a_{A,B,C})}} &
\aDoc (A\otimes (B\otimes C))}\]
\[\xymatrix{
\aDoc 1 \times A \ar[d]_{(\lambda_Q)_A} \ar[r]^{\ple{\iota,\id{}}} 
& \aDoc I \times \aDoc A \ar[d]^{\bullet_{I,A}}\\
\aDoc A \ar[r]^-{\DReIdx{\aDoc}{(l_A)}} & \aDoc (I\otimes A)}\qquad
\xymatrix{\aDoc A \times 1 \ar[d]_{(\rho_Q)_A} \ar[r]^{\ple{\id{},\iota}} 
& \aDoc A \times \aDoc I \ar[d]^{\bullet_{A,I}}\\
  \aDoc A \ar[r]^-{\DReIdx{\aDoc}{(r_A)}}& \aDoc (A\otimes I)}\]
\[\xymatrix{\aDoc A \times \aDoc B \ar[d]_{(\sigma_Q)_{A,B}} \ar[r]^{\bullet_{A,B}} 
& \aDoc (A\otimes B) \ar[d]^{\DReIdx{\aDoc}{(s_{A,B})}} \\
\aDoc B \times \aDoc A \ar[r]^{\bullet_{B,A}} & \aDoc(B\otimes A)}\]
Note that a primary doctrine \fun{\Doc}{\C\op}{\Pos} is a monoidal doctrine with 
\oneAr{\ple{\times,\sqcap}}{\Doc\times\Doc}{\Doc} and \oneAr{\ple{1,\top_1}}{\One}{\Doc}, where 
$1$ is the terminal object and $\top_1$ is the top element in $\Doc 1$, 
$\times$ is the binary product in the category and $\sqcap$ is defined, for all objects
$X,Y$ in \C, by 
$\sqcap_{X,Y} = \Land_{X\times Y}\circ(\DReIdx{\Doc}{\pi_1}\times\DReIdx{\Doc}{\pi_2})$,
where \fun{\pi_1}{X\times Y}{X} and \fun{\pi_2}{X\times Y}{Y} are the projections.

Now, consider a primary doctrine \Doc and a monoidal doctrine \aDoc.
An adjunction \ple{\Doc,\aDoc,\LAdj,\LAdjLift,\RAdj,\RAdjLift,\unit,\counit} is said to be \emph{monoidal} if $\LAdj$ and $\RAdj$ are lax monoidal functors and $\unit$ and $\counit$ are monoidal natural trasformations, that is, we have the following additional structure: 
\begin{itemize}
\item two 2-arrows \twoAr{u}{\ple{I,\iota}}{\ple{\LAdj\LAdjLift}\circ\ple{1,\top}} and \twoAr{\phi}{\ple{\otimes,\bullet}\circ (\ple{\LAdj,\LAdjLift}\times\ple{\LAdj,\LAdjLift})}{\ple{\LAdj,\LAdjLift}\circ\ple{\times,\sqcap}}, that is, \fun{u}{I}{\LAdj 1} and, for all objects $X,Y$ in \C, \fun{\phi_{X,Y}}{\LAdj X \otimes \LAdj Y}{\LAdj(X\otimes Y)} are arrows in \D, and 
\item two 2-arrows \twoAr{v}{\ple{1,\top}}{\ple{\RAdj,\RAdjLift}\circ\ple{I,\iota}} and \twoAr{\psi}{\ple{\times,\sqcap}\circ(\ple{\RAdj,\RAdjLift}\times\ple{\RAdj,\RAdjLift})}{\ple{\RAdj,\RAdjLift}\circ\ple{\times,\sqcap}}, that is, \fun{v}{1}{\RAdj I} and, for all objects $A,B$ in \D, \fun{\psi_{A,B}}{\RAdj A \times \RAdj B}{\RAdj(A\times B)} are arrows in \C, and 
\item the following diagrams commute
\[\xymatrix{
X\times Y \ar[d]_{\unit_X\times\unit_Y} \ar[rr]^{\id{X\times Y}}&&X\times Y \ar[d]^{\unit_{X\times Y}}\\
\RAdj\LAdj X \times \RAdj\LAdj Y \ar[r]^{\psi_{\LAdj X,\LAdj Y}} 
&\RAdj(\LAdj X \otimes \LAdj Y) \ar[r]^{\RAdj\phi_{X,Y}} & \RAdj\LAdj(X\times Y)}\]
\[\xymatrix{\LAdj\RAdj A\otimes\LAdj\RAdj B\ar[d]_{\counit_A\otimes\counit_B}\ar[r]^{\phi_{\RAdj A,\RAdj B}}&
\LAdj(\RAdj A \times \RAdj B) \ar[r]^{\LAdj\psi_{A,B}} 
&\LAdj\RAdj(A\otimes B) \ar[d]^{\counit_{A\otimes B}} \\
A\otimes B \ar[rr]^{\id{A\otimes B}} && A\otimes B}\]
\[\xymatrix{
1 \ar[d]_{\unit_1} \ar[r]^{v} & \RAdj I \ar[r]^{\RAdj u}& \RAdj\LAdj 1\\
\RAdj\LAdj 1 \ar[urr]_{\id{\RAdj\LAdj 1}} &&  &}\qquad
\xymatrix{I\ar[drr]_{\id{I}}\ar[r]^{u}&\LAdj1\ar[r]^{\LAdj v}&\LAdj\RAdj I\ar[d]^{\counit_I} \\ 
& & I}\] 
and the following inequalities on the fibres: 
\[\xymatrix{
\aDoc A \times \aDoc B \ar@{}[rdd]|{\le} \ar[r]^{\bullet_{A,B}} \ar[d]_{\RAdjLift_A\times\RAdjLift_B} & \aDoc(A\otimes B) \ar[dd]^{\RAdjLift_{A\otimes B}} \\ 
\Doc\RAdj A\times \Doc\RAdj B  \ar[d]_{\sqcap_{\RAdj A,\RAdj B}} & \\
\Doc (\RAdj A \times \RAdj B)  & \Doc \RAdj (A\otimes B) \ar[l]_{\DReIdx{\Doc}{\psi_{A,B}}} \\
1 \ar@{}[dr]|{\le} \ar[d]_{\top_1} \ar[r]^{\iota} & \aDoc I \ar[d]^{\RAdjLift_I} \\
\Doc 1  &  \Doc \RAdj I  \ar[l]_{\DReIdx{\Doc}{v}} 
}
\qquad
\xymatrix{
\Doc X \times \Doc Y \ar@{}[ddr]|{\le} \ar[r]^{\sqcap_{X,Y}} \ar[d]_{\LAdjLift_X\times\LAdjLift_Y} & \Doc(X\times Y) \ar[dd]^{\LAdjLift_{X\times Y}} \\ 
\aDoc\LAdj X\times \aDoc\LAdj Y  \ar[d]_{\bullet_{\LAdj X,\LAdj Y}} & \\
\aDoc (\LAdj X \otimes \LAdj Y) & \aDoc \LAdj (X\times Y) \ar[l]_{\DReIdx{\aDoc}{\phi_{X,Y}}} \\
1 \ar@{}[dr]|{\le} \ar[d]_{\iota} \ar[r]^{\top_1} & \Doc 1 \ar[d]^{\LAdjLift_1} \\ 
\aDoc I & \aDoc \LAdj 1 \ar[l]_{\DReIdx{\aDoc}{u}} 
} \]
\end{itemize} 
From general results about monoidal adjunctions between categories, we know that $u$ and
$\phi$ are (natural) isos. Hence the inequalities on the left-hand side are
equalities, that is, those diagrams commute.

Consider now the doctrine \fun{\aDoc\LAdj\op}{\C\op}{\Pos}. 
By \refToCor{adj-to-modality}, there is \acomonadic  operator
\TNat{\Lbang}{\aDoc\LAdj\op}{\aDoc\LAdj\op} defined as
$\Lbang = \LAdjLift \cdot (\Doc\unit\op) \cdot \RAdjLift\LAdj\op$. 
However, in this richer context, $\aDoc\LAdj\op$ has a richer structure. 
First of all \C has finite products, hence, for each object $X$ in \C, there are arrows
\fun{\zeta}{X}{1} and \fun{\Delta_X}{X}{X\times X} natural in $X$. 
Then, we can define a monoid structure on $\aDoc\LAdj\op X$ as the two composite arrows 
\[\xymatrix{1\ar@<-1ex>@/_/[rrr]_-{e_X} \ar[r]^{\iota}&\aDoc I \ar[r]^{\DReIdx{\aDoc}{u^{-1}}}
&\aDoc(\LAdj 1) \ar[r]^{\DReIdx{\aDoc}{\LAdj\zeta_X}} &\aDoc(\LAdj X)}\]
\[\xymatrix@C=7ex{\aDoc(\LAdj X)\times\aDoc(\LAdj X)\ar@<-1ex>@/_/[rrr]_-{\ast_X}
\ar[r]^{\bullet_{\LAdj X,\LAdj X}} & \aDoc (\LAdj X\otimes \LAdj X) 
\ar[r]^{\DReIdx{\aDoc}{\phi_{X,X}^{-1}}}  
& \aDoc(\LAdj(X\times X)) \ar[r]^{\DReIdx{\aDoc}{\LAdj \Delta_X}} & \aDoc(\LAdj X).}\]
It follows that \ple{\aDoc\LAdj\op X,\ast_X,e_X} is a commutative monoid and that such
structure is preserved by reindexing.  
This structure interprets the multiplicative conjunction of linear logic and its unit. 
To ensure that $\Lbang$ correctly interprets the ``bang'' modality of linear logic,
four properties, in addition to those of \comonadic operators, are required to hold:
for each object $X$ in \C and $\alpha, \beta  \in \aDoc(\LAdj X)$, 
\[\begin{array}{ll@{\qquad}ll}
(1)&\Lbang_X\alpha \order e_X&
(2)&\Lbang_X\alpha \order \Lbang_X\alpha \ast_X \Lbang_X\alpha\\[1ex]
(3)&e_X \order \Lbang_X e_X&
(4)&\Lbang_X \alpha \ast_X \Lbang_X \beta \order \Lbang_X (\alpha \ast_X \beta).
\end{array}\]
\begin{enumerate}
\item Note that $\DReIdx{\Doc}{\unit_X}(\RAdjLift_{\LAdj X}(\alpha)) \in \Doc X$, which is
an inf-semilattice with top element $\top_X$, hence
$\DReIdx{\Doc}{\unit_X}(\RAdjLift_{\LAdj X}(\alpha)) \order \top_X=
\DReIdx{\Doc}{\zeta_X}(\top_1)$, because reindexing preserves the inf-semilattice
structure.
Therefore, we get 
$\Lbang_X \alpha = \LAdjLift_X(\DReIdx{\Doc}{\unit_X}(\RAdjLift_{\LAdj X}(\alpha))) \order
\LAdjLift_X(\DReIdx{\Doc}{\zeta_X}(\top_1))=
\DReIdx{\aDoc}{\LAdj\zeta_X}(\LAdjLift_1(\top_1)) = e_X$, by naturality of $\LAdjLift$ and
one of the diagrams above.  
\item Again, note that $\DReIdx{\Doc}{\unit_X}(\RAdjLift_{\LAdj X}(\alpha)) \in \Doc X$,
which is an inf-semilattice, hence $\DReIdx{\Doc}{\unit_X}(\RAdjLift_{\LAdj X}(\alpha))
\order \DReIdx{\Doc}{\unit_X}(\RAdjLift_{\LAdj X}(\alpha)) \Land_X
\DReIdx{\Doc}{\unit_X}(\RAdjLift_{\LAdj X}(\alpha))$.  
Since $\pi_i \circ \Delta_X = \id{X}$, using naturality of $\Land$, we get 
\[\begin{split}
\DReIdx{\Doc}{\unit_X}(\RAdjLift_{\LAdj X}(\alpha)) 
&\order \DReIdx{\Doc}{\Delta_X}( \DReIdx{\Doc}{\pi_1}(\DReIdx{\Doc}{\unit_X}(\RAdjLift_{\LAdj X}(\alpha))) \Land_{X\times X} \DReIdx{\Doc}{\pi_2}(\DReIdx{\Doc}{\unit_X}(\RAdjLift_{\LAdj X}(\alpha))) ) \\
&= \DReIdx{\Doc}{\Delta_X}( \DReIdx{\Doc}{\unit_X}(\RAdjLift_{\LAdj X}(\alpha)) \sqcap_{X,X} \DReIdx{\Doc}{\unit_X}(\RAdjLift_{\LAdj X}(\alpha)) )
\end{split}\]
Therefore, applying $\LAdjLift_X$ and using one of the diagrams above we get 
\[\begin{split}
\Lbang_X\alpha &= \LAdjLift_X(\DReIdx{\Doc}{\unit_X}(\RAdjLift_{\LAdj X}(\alpha))) \\
&\order \LAdjLift_X(\DReIdx{\Doc}{\Delta_X}( \DReIdx{\Doc}{\unit_X}(\RAdjLift_{\LAdj X}(\alpha)) \sqcap_{X,X} \DReIdx{\Doc}{\unit_X}(\RAdjLift_{\LAdj X}(\alpha)) ) \\
&= \DReIdx{\aDoc}{\LAdj\Delta_X}( \LAdjLift_{X\times X}( \DReIdx{\Doc}{\unit_X}(\RAdjLift_{\LAdj X}(\alpha)) \sqcap_{X,X} \DReIdx{\Doc}{\unit_X}(\RAdjLift_{\LAdj X}(\alpha)) ) ) \\ 
&= \LAdjLift_X(\DReIdx{\Doc}{\unit_X}(\RAdjLift_{\LAdj X}(\alpha))) \ast_X \LAdjLift_X(\DReIdx{\Doc}{\unit_X}(\RAdjLift_{\LAdj X}(\alpha))) \\
&= \Lbang_X\alpha \ast_X \Lbang_X\alpha
\end{split}\]
\item By one of the diagrams above, naturality of $\LAdjLift$ and the fact that reindexing in $\Doc$ preserves the inf-semilattice structure, we have $e_X = \LAdjLift_X(\top_X)$. 
Furthermore, since \twoAr{\unit}{\ple{\Id{},\id{}}}{\ple{\RAdj\LAdj,(\RAdjLift\LAdj\op)\LAdjLift}} is a 2-arrow in \IdxPos, we get 
\[ e_X = \LAdjLift_X(\top_X) \order \LAdjLift_X(\DReIdx{\Doc}{\unit_X}(\RAdjLift_{\LAdj X}(\LAdjLift_X(\top_X)))) = \Lbang_X e_X \] 
\item Using the diagrams above and the definitions of $\ast_X$ and $\Lbang_X$ we get 
\[\begin{split} 
\Lbang_X \alpha \ast_X \Lbang_X \beta &= (\LAdjLift_X(\DReIdx{\Doc}{\unit_X}(\RAdjLift_{\LAdj X}(\alpha)))) \ast_X (\LAdjLift_X(\DReIdx{\Doc}{\unit_X}(\RAdjLift_{\LAdj X}(\beta)))) \\ 
&= \DReIdx{\aDoc}{\LAdj\Delta_X} (\LAdjLift_{X\times X}(
  \DReIdx{\Doc}{\unit_X}(\RAdjLift_{\LAdj X}(\alpha)) \sqcap_{X,X} \DReIdx{\Doc}{\unit_X}(\RAdjLift_{\LAdj X}(\beta)) 
)) \\
&= \LAdjLift_X(\DReIdx{\Doc}{\Delta_X}(\DReIdx{\Doc}{(\unit_X\times\unit_X)}(
  \RAdjLift_{\LAdj X}(\alpha) \sqcap_{\RAdj\LAdj X,\RAdj\LAdj X} \RAdjLift_{\LAdj X}(\beta)
)))  \\
&\order \LAdjLift_X(\DReIdx{\Doc}{\Delta_X}(\DReIdx{\Doc}{(\unit_X\times\unit_X)}(
  \DReIdx{\Doc}{\psi_{\LAdj X,\LAdj X}}(\RAdjLift_{\LAdj X \otimes\LAdj X}(\alpha \bullet_{\LAdj X,\LAdj X} \beta))  
))) 
\end{split}\] 
From one of the diagrams above, we have $\psi_{\LAdj X,\LAdj X} \circ (\unit_X \times\unit_X) = \RAdj\phi_{X,X}^{-1} \circ \unit_{X\times X}$, hence we get 
\[\begin{split}
\Lbang_X \alpha \ast_X \Lbang_X \beta  
&\order \LAdjLift_X(\DReIdx{\Doc}{\Delta_X}(\DReIdx{\Doc}{(\unit_X\times\unit_X)}(
  \DReIdx{\Doc}{\psi_{\LAdj X,\LAdj X}}(\RAdjLift_{\LAdj X \otimes\LAdj X}(\alpha \bullet_{\LAdj X,\LAdj X} \beta)) 
)))  \\ 
&= \LAdjLift_X(\DReIdx{\Doc}{\Delta_X}(\DReIdx{\Doc}{\unit_{X\times X}}(
  \DReIdx{\Doc}{\RAdj\phi_{X,X}^{-1}}(\RAdjLift_{\LAdj X \otimes\LAdj X}(\alpha \bullet_{\LAdj X,\LAdj X} \beta))
)))  \\ 
&= \LAdjLift_X(\DReIdx{\Doc}{\unit_X}(
  \RAdjLift_{\LAdj X}(\DReIdx{\aDoc}{\LAdj \Delta_X}(\DReIdx{\aDoc}{\phi_{X,X}^{-1}}(\alpha \bullet_{\LAdj X,\LAdj X} \beta)))
))  \\ 
&= \Lbang_X(\alpha\ast_X\beta) 
\end{split}\] 
\end{enumerate}

\end{document}